\documentclass[11pt]{article}

\usepackage[numbers]{natbib}
\usepackage{amsthm,amsmath,mathrsfs}
\usepackage{amssymb,amsthm,amsmath}
\usepackage[dvips]{geometry}
\usepackage{amscd}
\usepackage{amsfonts}
\usepackage[latin1]{inputenc}
\usepackage[all] {xy}
\usepackage[normalem]{ulem}
\usepackage{hyperref}
\usepackage{mathtools}
\usepackage[dvips]{geometry}
\usepackage{graphicx}
\usepackage{subcaption}
\usepackage[nottoc]{tocbibind}

\graphicspath{ {images/}}

\newtheorem{theorem}{Theorem}[section]

\newtheorem*{theorema}{Main Theorem Restated}
\newtheorem*{maintheorem}{Main Theorem}
\newtheorem{proposition}[theorem]{Proposition}
\newtheorem{corollary}[theorem]{Corollary}
\newtheorem{lemma}[theorem]{Lemma}

\newtheorem{definition}[theorem]{Definition}

\newtheorem{remark}[theorem]{Remark}

\newtheorem{claim}{Claim}

\newcommand{\T}{\mathbb{T}}
\newcommand{\field}[1]{\mathbb{#1}}

\newcommand{\Tor}{\field{T}}

\usepackage{ulem}

\newcommand{\comment}[1]{}
\newcommand{\Z}{\mathbb{Z}}
\newcommand{\N}{\mathbb{N}}
\newcommand{\R}{\mathbb{R}}
\newcommand{\C}{\mathscr{C}}

\newcommand{\mrworld}{\mathrm{Diff}^r_{Leb}(M)}

\newcommand{\dworld}{\mathrm{Diff}^2_{Leb}(\mathbb{T}^4)}

\newcommand{\address}{{
  \bigskip
  \footnotesize

  \textsc{CNRS-Laboratoire de Math\'ematiques d'Orsay, UMR 8628, Universit\'e Paris-Sud 11, Orsay Cedex 91405, France } \par\nopagebreak
  \textsc{Instituto de Matem\'atica, Universidade Federal do Rio de Janeiro, P.O. Box 68530, 21945-970, Rio de Janeiro Brazil}\par\nopagebreak
  \textit{E-mail:} \texttt{davi.obata@math.u-psud.fr}
}}

\title{On the Stable Ergodicity of Berger-Carrasco's example}
\author{Davi Obata \footnote{D.O. was supported by the ERC project 692925 NUHGD.}}
\date{\today}

\begin{document}

\maketitle
\begin{abstract}
We prove the stable ergodicity of an example of a volume-preserving, partially hyperbolic diffeomorphism introduced by Berger and Carrasco in \cite{bergercarrasco2014}. This example is robustly non-uniformly hyperbolic, with two dimensional center, almost every point has both positive and negative Lyapunov exponents along the center direction and does not admit a dominated splitting of the center direction. The main novelty of our proof is that we do not use accessibility.
\end{abstract}
\setcounter{tocdepth}{1}
\tableofcontents

\section{Introduction}
\label{introduction}

Let $M$ be a smooth compact riemannian manifold and let $\nu$ be a Borel probability measure on $M$. Given a measurable transformation $f:M \to M$ that preserves $\nu$, we say that $f$ is ergodic with respect to $\nu$ if every invariant measurable set has either zero or full measure. Ergodicity means that from the probabilistic point of view the system cannot be decomposed into invariant smaller parts. In our scenario, $f$ is ergodic if and only if for every continuous function $\varphi:M \to M$, for $\nu$-almost every point $p\in M$ it is verified  
\[
\displaystyle \lim_{n\to +\infty} \frac{1}{n} \sum_{j=0}^{n-1} \varphi \circ f^j(p)= \int_M \varphi d\nu.
\]

In 1939, Hopf introduced in \cite{hopf39} an argument to prove that the geodesic flow on compact surfaces with constant negative curvature is ergodic with respect to the Liouville measure. Many years later, Anosov \cite{anosov67}, Anosov and Sinai \cite{anosovsinai67} used the Hopf argument to prove ergodicity of hyperbolic systems that preserve a smooth measure. A diffeomorphism is hyperbolic, or Anosov, if its tangent bundle decomposes into two invariant subbundles, one is contracted and the other one is expanded exponentially fast by the action of the derivative. Hyperbolicity was the key property that allowed them to use the Hopf argument in these settings.

Since then several works extended the Hopf argument to more general settings, namely non-uniformly hyperbolic and partially hyperbolic systems.

For a $C^1$-diffeomorphism $f$ and an invariant measure $\nu$, Kingman's ergodic theorem implies that for $\nu$-almost every point $p\in M$ and for every $v\in T_pM-\{0\}$ the following limit exists
\begin{equation}
\label{lyapunovdefinition}
\displaystyle \lambda(p,v)=\lim_{n \to \pm \infty} \frac{1}{n} \log \|Df^n(p).v\| .
\end{equation}

Oseledets' theorem states that $\lambda(p,.)$ can take at most $\mathrm{dim}(M)$ different values. Such numbers are called Lyapunov exponents. A $f$-invariant measure $\nu$ is non-uniformly hyperbolic for $f$ if for $\nu$-almost every point, every Lyapunov exponent is non zero.  

In \cite{pesin77}, Pesin uses the Hopf argument to prove that if $\nu$ is a smooth, non-uniformly hyperbolic measure and $f$ is a $C^{1+\alpha}$-diffeomorphism then $\nu$ has at most countably many ergodic components.

A diffeomorphism $f$ is partially hyperbolic if there is a $Df$-invariant decomposition $TM = E^{ss} \oplus E^c \oplus E^{uu}$, such that $Df|_{E^{ss}}$ contracts, $Df|_{E^{uu}}$ expands and the behaviour of $Df|_{E^c}$ is bounded by the contraction of $E^{ss}$ and the expansion of $E^{uu}$. See section \ref{preliminary} for a precise definition.

A key property for discussing the ergodicity of partially hyperbolic systems is the accessibility. A partially hyperbolic system is accessible if any two points can be joined by a curve which is a concatenation of finitely many curves, each of them being contained in a stable or an unstable leaf. 

There are several works that use accessibility to extend the Hopf argument and prove ergodicity, see for instance \cite{brinpesin74}, \cite{gps94}, \cite{pughshub00}, \cite{bdp02}, \cite{bw10} and \cite{hhtu11}. Most proofs of the ergodicity for partially hyperbolic systems uses accessibility. Several of the extensions of the Hopf argument for accessible partially hyperbolic diffeormorphisms allow vanishing Lyapunov exponents along the center direction.

Berger and Carrasco introduced in \cite{bergercarrasco2014} an example of a volume-preserving, partially hyperbolic diffeomorphism which is non-uniformly hyperbolic. This example has a two dimensional center bundle and Lebesgue almost every point has both positive and negative Lyapunov exponent in the center direction. Furthermore, the properties of this example are $C^2$-robust. It is not known if this example is accessible or not.

\begin{definition}
\label{stableergodicity}
A volume-preserving diffeomorphism $f$ is {\bf $C^2$-stably ergodic} if it admits a $C^2$-neighborhood such that any volume-preserving diffeomorphism inside this neighborhood is ergodic.  
\end{definition}

In this paper we prove the following theorem.

\begin{maintheorem}
The Berger-Carrasco's example is $C^2$-stably ergodic.
\end{maintheorem}

We stress two features of our work that distinguishes it from the rest of the previous works about ergodicity of partially hyperbolic diffeomorphisms:
\begin{itemize}
\item The stable ergodicity with mixed behaviour along the center direction and that does not admit a dominated splitting of the center direction (as a strengthening of \cite{bergercarrasco2014});
\item A proof of stable ergodicity that does not uses accessibility.
\end{itemize}

We explain a couple points on why on definition \ref{stableergodicity} we use a $C^2$-neighborhood instead of a $C^1$-neighborhood, which is the one usually used to define stable ergodicity, see for instance \cite{hhtu11}. First, the techniques we use depend on the uniform control of $C^2$-norms in a neighborhood. Second, it is not possible to have the mixed behaviour along the center for every volume-preserving, $C^2$-diffeomorphism in a $C^1$-neighborhood of Berger-Carrasco's example. This is due to theorem A' in \cite{avilacrovisierwilkinson2017}, which implies that arbitrarily $C^1$-close to Berger-Carrasco's example there is a volume-preserving, $C^2$-diffeomorphism which is stably ergodic and whose Lyapunov exponents along the center have the same sign.    

From now on we will denote the normalized Lebesgue measure of a manifold by $Leb$ and by $\mrworld$ the set of $C^r$-diffeomorphisms that preserve the Lebesgue measure.
 
\subsection*{Berger-Carrasco's example and the precise statement of the main theorem}

For $N \in \R$ we denote by $s_N(x,y) = (2x-y + N\sin(x), x)$ the standard map on $\T^2 = \R^2/ 2\pi \Z^2$. For every $N$ the map $s_N$ preserves the Lebesgue measure induced by the usual metric of $\T^2$. 

This map is related to several physical problems, see for instance \cite{chirikovstandard}, \cite{i80chaos} and \cite{s95chaos}. 

It is conjectured that for $N \neq 0$ the map $s_N$ has positive entropy for the Lebesgue measure, see \cite{sinaiergodictheory} page $144$. By Pesin's entropy formula, see \cite{pesin77} Theorem $5.1$, this is equivalent to the existence of a set of positive Lebesgue measure and whose points have a positive Lyapunov exponent. The existence of those sets is not known for any value of $N$. See \cite{young17standard}, \cite{duarte94standard} and \cite{goro12standard} for some results related to this conjecture.

Let $A\in SL(2,\Z)$ be a hyperbolic matrix which defines an Anosov diffeomorphism on $\Tor^2$, let $P_x: \T^2 \to \T^2$ be the projection on the first coordinate of $\T^2$, this projection is induced by the linear map of $\R^2$, which we will also denote by $P_x$, given by $P_x(a,b) = (a,0)$. In a similar way define $P_y:\Tor^2 \to \T^2$ the projection on the second coordinate of the torus.

Consider the torus $\T^4= \T^2 \times \T^2$ and represent it using the coordinates $(x,y,z,w)$, where $x,y,z,w \in [0,2\pi)$. We may naturally identify a point $(z,w)$ on the second torus with a point $(x,y)$ on the first torus by taking $x=z$ and $y=w$. For each $N \geq 0$ define
$$
\begin{array}{rcccc}
f_N &:& \T^2\times \T^2 & \longrightarrow & \T^2\times\T^2\\
&&(x,y,z,w)& \mapsto & (s_N(x,y) + P_x\circ A^N(z,w),A^{2N}(z,w)),
\end{array}
$$
where the point $ A^N(z,w)$ on the second torus is being identified with the same point in the first torus as described previously.  

This diffeomorphism preserves the Lebesgue measure. For $N$ large enough it is a partially hyperbolic diffeomorphism, with two dimensional center direction given by $E^c = \R^2 \times \{0\}$. This type of system was considered by Berger and Carrasco in \cite{bergercarrasco2014}, where they proved the following theorem.

\begin{theorem}[\cite{bergercarrasco2014}, Theorem 1]
\label{bcexm}
There exist $N_0>0$ and $c>0$ such that for every $N\geq N_0$, for Lebesgue almost every point $m$ and for every $v\in \R^4$
$$
\displaystyle \lim_{n\to\infty} \left|\frac{1}{n}\log\|Df_N^n(m).v\|\right| > c\log N.
$$

Moreover, the same holds for any volume-preserving diffeomorphism in a $C^2$-neighborhood of $f_N$.

\end{theorem}

This theorem says that for $N$ large enough the system $f_N$ is non-uniformly hyperbolic. Indeed, along the center direction there is one positive and one negative Lyapunov exponent for Lebesgue almost every point.

We remark that Viana constructed in theorem B of \cite{vianamaps}, an example of a non-conservative partially hyperbolic diffeomorphism with similar properties as in Berger and Carrasco's example, meaning Lebesgue almost every point has a positive and a negative exponent in the center direction and there is no dominated splitting of the center, but in the dissipative case. The approach used by Berger and Carrasco has some similarities with Viana's approach, which is to consider ``unstable" curves and use combinatorial arguments to estimate the exponents over such a curve. 

\begin{definition}
\label{bernoulli}
Let $\nu$ be an invariant probability measure for $f$. We say that $(f,\nu)$ is Bernoulli if it is measurably conjugated to a Bernoulli shift. For volume-preserving diffeomorphisms, we say that $f$ is Bernoulli if $(f, Leb)$ is Bernoulli.
\end{definition}

The Bernoulli property is stronger than ergodicity. We can now give the precise statement of the main theorem.

\begin{theorema}
\label{maintheorem}
 For $N$ large enough $f_N$ is $C^2$-stably ergodic. Moreover, any volume-preserving diffeomorphism in a $C^2$-neighborhood of $f_N$ is Bernoulli. 
\end{theorema}

In order to prove this theorem we will need to obtain precise estimates on the size of the invariant manifolds in the center direction for certain points. For that we will need a better estimate of the center exponents, given by the following proposition.

\begin{proposition}
\label{estimateprop}
For every $\delta\in(0,1)$, there exists $N_0=N_0(\delta)$ such that for every $N\geq N_0$ there is a $C^2$-neighborhood $\mathcal{U}_N$ of $f_N$ in $\dworld$ with the following property. If $g\in \mathcal{U}_N$, then Lebesgue almost every point has a positive and a negative Lyapunov exponent in the center direction whose absolute value are greater than $(1-\delta)\log N$.
\end{proposition}

We remark that one can show that $f_N$ is $C^2$-approximated by stably ergodic diffeomorphisms with another approach. This approach uses accessibility, which can be obtained using the results in \cite{horitasambarino2017}, and the criteria of ergodicity in \cite{bw10}. Such approach does not use the non-uniform hyperbolicity of the system. 

\subsection*{Strategy of the proof}
The strategy of the proof has two parts. The first part is the construction of stable and unstable manifolds inside center leaves with precise estimates on its length and ``geometry". The second part is the global strategy to obtain the ergodicity.

For the first part, the main tool is to use the construction of stable manifolds for surface diffeomorphisms, given by Crovisier and Pujals in theorem $5$ of \cite{crovisierpujalsstronglydissipative}. In order to do that two ingredients are needed. The first is a good control of the Lyapunov exponents along the center direction so it verifies some inequality, see the beginning of section \ref{lowerbound} for a discussion. The second is to find sets with positive measure of points with good contraction and expansion for the Oseledecs splitting, for any ergodic component.

Proposition \ref{estimateprop} gives the control needed of the Lyapunov exponents. To prove proposition \ref{estimateprop}, we follow the proof of theorem \ref{bcexm}, given by Berger and Carrasco in \cite{bergercarrasco2014}, with the necessary adaptations to obtain a precise estimate of the Lyapunov exponents along the center. For the second ingredient, we use a version of the Pliss lemma, lemma \ref{pliss}. Following the construction of Crovisier and Pujals in \cite{crovisierpujalsstronglydissipative}, we obtain precise estimates of the length and the ``geometry" of stable and unstable curves inside center leaves, given by propositions \ref{sizemnfld} and \ref{stablesizenotfibered}. So far what is obtained with this construction is that any ergodic component of the Lebesgue measure has a set of points with positive measure having stable and unstable curves in the center leaves of uniform size and controlled ``geometry". That alone guarantees that there are at most finitely many ergodic components. 

For the global strategy there are also two ingredients, the estimate on the measure of points with good expansion and contraction, given by Pliss lemma, and the density of the orbit of almost every center leaf among the center leaves.

The estimate on the measure given by Pliss lemma is used to obtain points that spend a long time inside a region with good hyperbolicity. This together with the control on the length and ``geometry" of the stable and unstable curves inside the center leaves allows us to obtain points whose such curves are very large inside the center direction. The density of the orbit of almost every center leaf together with these large stable and unstable manifolds is then used to apply the Hopf argument and conclude the ergodicity. 

We remark that in this proof we use the Hopf argument for non-uniformly hyperbolic systems and not the version usually used for partially hyperbolic diffeomorphisms, see for instance \cite{bw10}.

\subsection*{Organization of the paper}

In section \ref{preliminary} we will introduce several tools that will be used in the proof. We will assume that proposition \ref{estimateprop} holds throughout sections \ref{intmnfld}, \ref{ergodicfiber}, \ref{stableergodicity} and \ref{bernoullisection}, which are dedicated to prove the {\it main theorem}. The proof of proposition \ref{estimateprop} is then given in Section \ref{exponents}. 

\subsection*{Acknowledgments}
The author would like to thank Sylvain Crovisier for all his patience and guidance with this project. The author thanks Alexander Arbieto for useful conversations specially regarding section \ref{bernoullisection}. The author also benefited from conversations with Martin Leguil, Frank Trujillo, Welington Cordeiro and Bruno Santiago.  

\section{Preliminaries}
\label{preliminary}

\subsection{General theory and results}

\subsubsection*{Partial hyperbolicity and foliations}

A $C^r$-diffeomorphism $f$, with $r\geq 1$, is {\bf partially hyperbolic} if the tangent bundle has a decomposition $TM = E^{ss} \oplus E^c \oplus E^{uu} $, there is a riemannian metric on $M$ and continuous functions $ \chi^*_-, \chi^*_+:M\to \R$, for $*=ss,c,uu$, with such that for any $m\in M$ 
\[
\chi^{ss}_+(m)<1< \chi^{uu}_-(m) \textrm{ and } \chi^{ss}_+(m) < \chi^c_-(m) \leq \chi^c_+(m) < \chi^{uu}_-(m),
\]
it also holds
\[\arraycolsep=1.2pt\def\arraystretch{2}
\begin{array}{rclclcl}
\chi^{ss}_-(m) &\leq & m(Df(m)|_{E^{ss}_m}) &\leq & \|Df(m)|_{E^{ss}_m}\|& \leq &  \chi^{ss}_+(m);\\
\chi^c_-(m) &\leq & m(Df(m)|_{E^c_m}) &\leq & \|Df(m)|_{E^c_m}\|& \leq &  \chi^c_+(m);\\
\chi^{uu}_-(m) &\leq & m(Df(m)|_{E^{uu}_m}) &\leq & \|Df(m)|_{E^{uu}_m}\|& \leq &  \chi^{uu}_+(m),
\end{array}
\]
where $m(Df(m)_{E^*_m}) = \|(Df(m)|_{E^*_m})^{-1}\|^{-1}$ is the co-norm of $Df(m)|_{E^*_m}$, for $* = ss,c,uu$. 	
If the functions in the definition of partial hyperbolicity can be taken constant, we say that $f$ is {\it absolutely partially hyperbolic}.	
 	
It is well known that the distributions $E^{ss}$ and $E^{uu}$ are uniquely integrable, that is, there are two unique foliations $\mathcal{F}^{ss}$ and $\mathcal{F}^{uu}$, with $C^r$-leaves, that are tangent to $E^{ss}$ and $E^{uu}$ respectively. For a point $p\in M$ we will denote by $W^{ss}(p)$ a leaf of the foliation $\mathcal{F}^{ss}$, we will call such leaf the strong stable manifold of $p$. Similarly we define the strong unstable manifold of $p$ and denote it by $W^{uu}(p)$.

\begin{definition}
\label{centerbunched}
A partially hyperbolic diffeomorphism is {\bf center bunched} if   
\[
\chi^{ss}_+(m) < \frac{\chi^c_-(m)}{\chi^c_+(m)} \textrm{ and }\frac{\chi^c_+(m)}{\chi^c_-(m)} < \chi^{uu}_-(m), \textrm{ for every $m\in M$}.
\]
\end{definition}

We denote $E^{cs} = E^s \oplus E^c$ and $E^{cu} = E^c \oplus E^u $.

\begin{definition}
\label{dynamicalcoherence}
A partially hyperbolic diffeomorphism $f$ is {\bf dynamically coherent} if there are two invariant foliations $\mathcal{F}^{cs}$ and $\mathcal{F}^{cu}$, with $C^1$-leaves, tangent to $E^{cs}$ and $E^{cu}$ respectively. From those two foliations one obtains another invariant foliation $\mathcal{F}^c = \mathcal{F}^{cs} \cap \mathcal{F}^{cu}$ that is tangent to $E^c$. We call those foliations the center-stable, center-unstable and center foliation.

\end{definition}

For any $R>0$ we write $W^*_R(p)$ to be the disc of size $R$ centered on $p$, for the Riemannian metric induced by the metric on $M$, contained in the leaf $W^*(p)$, for $*=ss,c,uu$. 

The definition below allows one to obtain higher regularity of the leaves of such foliations.

\begin{definition}
\label{rnormalhyperbolicity}
We say that a partially hyperbolic diffeomorphism $f$ is  {\bf $r$-normally hyperbolic} if for any $m\in M$
\[
\chi^s_+(m)< (\chi^c_-(m))^r \textrm{ and } (\chi^c_+(m))^r< \chi^u_-(m).
\] 
\end{definition}

\begin{definition}
Let $f$ and $g$ be partially hyperbolic diffeomorphisms of $M$ that are dynamically coherent, denote by $\mathcal{F}^c_f$ and $\mathcal{F}^c_g$ the center foliations. We say that $f$ and $g$ are {\bf leaf conjugated} if there is a homeomorphism $h:M\to M$ that takes leaves of $\mathcal{F}^c_f$ to leaves of $\mathcal{F}^c_g$ and such that for any $L \in \mathcal{F}^c_f$ it is verified
\[
h(f(L)) = g(h(L)).
\]

\end{definition}

One may study the stability of partially hyperbolic systems up to leaf conjugacy. Related to this there is a technical notion called {\bf plaque expansivity} which we will not define here, see chapter 7 of \cite{hps} for the definition. The next theorem is important for the theory of stability of partially hyperbolic systems.

\begin{theorem}[\cite{hps}, Theorem $7.4$]
\label{leafconjugacy}
Let $f: M \to M$ be a $C^r$-partially hyperbolic and dynamically coherent diffeomorphism. If $f$ is $r$-normally hyperbolic and plaque expansive then any $g:M \to M$ in a $C^r$-neighborhood of $f$ is partially hyperbolic and dynamically coherent. Moreover, $g$ is leaf conjugated to $f$ and the center leaves of $g$ are $C^r$-immersed manifolds. 
\end{theorem}

\begin{remark}
\label{continuitycoherent}
In the proof of the previous theorem, it is obtained for a fixed $R>0$, if $f$ satisfies the hypothesis of the theorem, then for $g$ sufficiently $C^r$-close to $f$, for any $m\in M$, $W^c_{f,R}(m)$ is $C^r$-close to $W^c_{g,R}(m)$. In particular, if the center foliation is uniformly compact then for every $g$ sufficiently $C^r$-close to $f$, for any $m\in M$, $W^c_f(m)$ is $C^r$-close to $W^c_g(m)$.
\end{remark}

It might be hard to check the condition of plaque expansiviness, but this is the case when the center foliation of a dynamically coherent, partially hyperbolic diffeomorphism is at least $C^1$, see Theorem $7.4$ of \cite{hps}. Usually the invariant foliations that appear in dynamics are only H\"older.

We can also obtain a better regularity for the center direction given by the following theorem, see section $4$ of \cite{pughshubwilkinson12} for a discussion on this topic.

\begin{theorem}
\label{regularitycenter}
Let $f$ be a $C^2$-partially hyperbolic diffeomorphism and let $\alpha>0$ be a number such that for every $m\in M$ it is verified
\[
\chi^s_+(m)< \chi^c_-(m) (\chi^s_-(m))^{\alpha} \textrm{ and } \chi^c_+(m) (\chi^u_+(m))^{\alpha} < \chi^u_-(m),
\]
then $E^c$ is $\alpha$-H\"older.
\end{theorem}

\subsubsection*{Pesin's theory}

Let $f$ be a $C^1$-diffeomorphism, for a number $\lambda\in \R$ define $E^{\lambda}_p$ to be the subspace of the vector zero united with all vectors $v\in T_pM-\{0\}$ such that the number $\lambda(p,v)=\lambda$, where $\lambda(p,v)$ is the number defined in (\ref{lyapunovdefinition}).

We say that a set $R$ has full probability if for any $f$-invariant probability measure $\nu$ it is verified that $\nu(R)=1$.  The following theorem is known as the Oseledets theorem.

\begin{theorem}[\cite{barreirapesinbook}, Theorems $2.1.1$ and $2.1.2$]
\label{oseledets}
For any $C^1$-diffeomorphism $f$, there is a set $\mathcal{R}$ of full probability, such that for every $\varepsilon>0$ it exists a measurable function $C_{\varepsilon}: \mathcal{R} \to (1, +\infty)$ with the following properties:
\begin{enumerate}
\item for any $p\in \mathcal{R}$ there are numbers $s(p)\in \N$, $\lambda_1(p) < \cdots < \lambda_{s(p)}(p)$ and a decomposition $T_pM = E^{1}_p \oplus \cdots \oplus E^{s(p)}_p$;\\
\item $s(f(p)) = s(p)$, $\lambda_i(f(p)) = \lambda_i(p)$ and $Df(p).E^{i}_p= E^{i}_{f(p)}$, for every $i= 1, \cdots, s(p)$;
\item for every $v\in E^{i}_p- \{0\}$ and $n\in \Z$ 
\[
C_{\varepsilon}(p)^{-1} e^{n.(\lambda_i(p)-\varepsilon)} \leq \frac{\|Df^n(p).v\|}{\|v\|} \leq C_{\varepsilon}(p) e^{n.(\lambda_i(p) + \varepsilon)} \textrm{ and } \lambda(p,v) = \lambda_i(p);
\]
\item the angle between $E^{i}_p$ and $ E^{j}_p$ is greater than $C_{\varepsilon}(p)^{-1}$, if $i\neq j$;
\item $C_{\varepsilon}(f(p)) \leq e^{\varepsilon} C_{\varepsilon}(p)$.
\end{enumerate} 
\end{theorem} 

We call the set $\mathcal{R}$ the set of regular points. For a fixed $\varepsilon>0$ and each $l\in \N$ we define the {\bf Pesin block} 
\begin{equation}
\label{pesinblock}
\mathcal{R}_{\varepsilon,l} = \{p\in \mathcal{R}: C_{\varepsilon}(p) \leq l\}.
\end{equation} 
We have the following decomposition
\begin{equation}
\label{unionpesinblocks}
\mathcal{R} = \displaystyle \bigcup_{l\in \N} \mathcal{R}_{\varepsilon,l}.
\end{equation}

A point $p\in \mathcal{R}$ has $k$ negative Lyapunov exponents if 
\[
\displaystyle \sum_{i: \lambda_i(p) <0} dim (E^i_p) = k.
\] 
Similarly for positive or zero Lyapunov exponents. From now on, we assume that $\nu$ is a $f$-invariant measure, not necessarily ergodic, and there are numbers $k$ and $l$ such that $\nu$-almost every point $p\in \mathcal{R}$ has $k$ negative and $l$ positive Lyapunov exponents. 

For a regular point we write 
\begin{equation}
\label{oseledecsdirection}
E^s_p = \displaystyle \bigoplus_{i: \lambda_i(p)<0} E^i_p \textrm{ and } E^u_p = \bigoplus_{i: \lambda_i(p)>0} E^i_p.
\end{equation}

\begin{definition}
\label{pesinmanifolddef}
For $f$ a $C^{2}$ diffeomorphism the {\bf stable Pesin manifold} of the point $p\in \mathcal{R}$ is
\[
W^s(p) =\{ q\in M: \displaystyle \limsup_{n\to +\infty} \frac{1}{n} \log d(f^n(p), f^n(q)) <0 \}.
\]
Similarly one defines the {\bf unstable Pesin manifold} as
\[
W^u(p) = \{ q\in M: \displaystyle \limsup_{n\to +\infty} \frac{1}{n} \log d(f^{-n}(p), f^{-n}(q)) <0 \}.
\]
\end{definition}

\begin{remark}
\label{osedetspesin}
If $f$ is also partially hyperbolic, with $TM = E^{ss} \oplus E^c \oplus E^{uu}$ then the Oseledets splitting refines the partial hyperbolic splitting. This means that for a regular point $p\in \mathcal{R}$, there are numbers $1\leq l_1< l_2 < s(p)$ such that 
\[
E^{ss}_p = \displaystyle \bigoplus_{i=1}^{l_1} E^i_p, \textrm{ } E^c_p = \bigoplus_{i=l_1+1}^{l_2} E^i_p \textrm{ and } E^{uu}_p = \bigoplus_{i=l_2+1}^{s(p)} E^i_p.
\]

This follows from a standard argument similar to the proof of the unicity of dominated splittings, see section $B.1.2$ from \cite{bonattidiazvianabook}. It also holds that for any regular point $p$, $E^{ss}_p \subset E^s_p$ and $E^{uu}_p \subset E^u_p$.
\end{remark}

Pesin's manifolds are immersed submanifolds, see section $4$ of \cite{pesin77}. A difficulty that appears is that such submanifolds in general do not vary continuously with the point, but they vary continuously on Pesin blocks. Let us make this more precise. Define $W_{loc}^s(p)$ to be the connected component $D^s(p)$ of $W^s(p) \cap B(p,r)$ containing $p$, such that $\partial D^s(p) \subset \partial B(p,r)$ and $r>0$ is a small fixed number depending only on $\varepsilon>0$ and $l\in \N$. 

\begin{theorem}[\cite{pesin77}, Theorems $4.1$ and $4.2$]
\label{pesincontinuity}
Let $f:M\to M$ be a $C^2$-diffeomorphism preserving a smooth measure $\nu$ and suppose that $\nu$-almost every regular point $p$ has the same number of negative and positive Lyapunov exponents. For each $l>1$, $\varepsilon>0$ small and $p\in \mathcal{R}_{\varepsilon,l}$, it is verified:
\begin{enumerate}
\item $W^s_{loc}(p)$ contains a disc centered at $p$ and tangent to $\displaystyle E^s_p$;\\
\item $p \mapsto W^s_{loc}(p)$ varies continuously in the $C^1$-topology over $\mathcal{R}_{\varepsilon,l}$. 
\end{enumerate}
\end{theorem}

A partition $\xi$ of $M$ is measurable with respect to a probability measure $\nu$, if up to a set of $\nu$-zero measure, the quotient $M/\xi$ is separated by a countable number of measurable sets. Denote by $\hat{\nu}$ the quotient measure in $M/\xi$.

By Rokhlin's desintegration theorem \cite{rokhlin1}, for a measurable partition $\xi$, there is set of conditional measures $\{\nu_D^{\xi}: D\in \xi\}$ such that for $\hat{\nu}$-almost every $D\in \xi$ the measure $\nu_D^{\xi}$ is a probability measure supported on $X$, for each measurable set $B\subset M$ the application $D \mapsto \nu^{\xi}_D(B)$ is measurable and it holds
\begin{equation}
\label{desintegration}
\nu(B) = \displaystyle \int_{M/\xi} \nu_D^{\xi}(B) d\hat{\nu}(D).
\end{equation}

Fix $\mathcal{R}_{\varepsilon,l}$ a Pesin block. For $p\in \mathcal{R}_{\varepsilon,l}$ and for $\rho>0$ small, define $B_s(p,\rho)$ as the union of the local stable pesin manifolds of the points $y\in B(p,\rho) \cap \mathcal{R}_{\varepsilon,l}$. Consider the measure $\nu_{p,\rho}= \nu|_{B_s(p,\rho)}$ and the measurable partition $\xi_s$ given by the partition of $B_s(p,\rho)$ by local stable Pesin manifolds. For such a partition let $\{\nu_{p,\rho,D}^{\xi_s}: D\in \xi_s\}$ be the set of conditional measures of the desintegration of $\nu_{p,\rho}$ with respect to $\xi_s$.

\begin{definition}
\label{absolutecontinuous1}
The measure $\nu$ has {\bf absolute continuous conditional measures on stable manifolds} if for every Pesin block $\mathcal{R}_{\varepsilon,l}$, every $\rho>0$ small enough, for $\hat{\nu}_{p,\rho}$-almost every $X\in \xi_s$, the measure $\nu_{p,\rho,D}^{\xi_s}$ is equivalent to the Lebesgue measure of a local stable Pesin manifold.
\end{definition}

Take $p\in \mathcal{R}$ and let $T_1$ and $T_2$ be two disks transverse to $W^s(p)$ close to $p$. We can define the holonomy map related to these disks as the map $H$ defined on a subset of $T_1\cap \mathcal{R}$, consisting of the points $q$ such that $W^s_{loc}(q)$ intersects transversely $T_2$. Recall that we are assuming that the number of negative and positive Lyapunov exponents are the same $\nu$-almost everywhere.

\begin{definition}
\label{abscontinuitystablepartition}
We say that the stable partition is {\bf absolutely continuous} if all holonomy maps are measurable and take sets with zero Lebesgue measure of $T_1$ to into sets of zero Lebesgue measure of $T_2$.
\end{definition}

Analogously we define all the above for the unstable partition.

\begin{theorem}[\cite{pesin77}, Theorem $4.4$]
\label{absolutecontinuouspesin}
Let $f$ be a $C^2$-diffeomorphism preserving a smooth measure $\nu$, then the stable and unstable partitions are absolutely continuous.
\end{theorem}

\begin{remark}
\label{absoluteremark}
This theorem implies that $\nu$ has absolute continuous conditional measures with respect to the stable, or unstable, manifolds, see theorem $5.11$ in \cite{barreirapesinbook}. In particular, a Fubini-like formula (\ref{desintegration}) holds locally.
\end{remark}

The notion of absolute continuity also makes sense for foliations, but for the holonomy maps of the foliation. The strong stable foliation $\mathcal{F}^{ss}$ of a $C^2$-partially hyperbolic diffeomorphism is absolutely continuous, see \cite{anosov67}. 

Usually the partition by strong stable leaves, given by the foliation $\mathcal{F}^{ss}$, is not measurable. In a foliated chart $U$, one may consider the restricted foliation $\mathcal{F}^{ss}|_U$ and the partition by strong stable leaves forms a measurable partition of $U$. Thus one can disintegrate a smooth measure locally along such foliation. The absolute continuity of the strong stable foliation implies that the conditional measures of this disintegration are equivalent to the Lebesgue measure of these manifolds, in particular a Fubini-like formula also holds, see \cite{pughvianawilkinson2007} for a discussion.

Recall that a $f$-invariant measure $\nu$ is non-uniformly hyperbolic if for $\nu$-almost every point all Lyapunov exponents are non-zero.

\begin{theorem}[\cite{pesin77}, Theorems $7.2$ and $8.1$]
\label{ergodiccomponentspesin}
Let $f$ be a $C^2$-diffeomorphism preserving a smooth measure $\nu$. If $\nu$ is non-uniformly hyperbolic then there are at most countably many ergodic components of $\nu$, that is, 
\[
\nu = \displaystyle \sum_{i\in \N} c_i \nu_i,
\]
 where $c_i \geq 0$, $\displaystyle \sum_{i\in \N} c_i = 1$, each $\nu_i$ is a $f$-invariant ergodic probability measure and if $i\neq j$ then $\nu_i \neq \nu_j$. Moreover, for each $i\in \N$, there is $k_i\in \N$ such that 
\[
\nu_i = \frac{1}{k_i} \displaystyle \sum_{j=1}^{k_i} \nu_{i,j},
\]
where each $\nu_{i,j}$ is a $f^{k_i}$-invariant probability measure, the system $(f^{k_i},\nu_{i,j})$ is Bernoulli and $\nu_{i,j} \neq \nu_{i,j}$ if $j \neq l$. Furthermore, $f$ permutes the measures $\nu_{i,j}$, that is, $f_*(\nu_{i,j}) = \nu_{i,j+1}$ for $j=1, \cdots, k_i-1$ and $f_*(\nu_{i,k_i}) = \nu_{i,1}$, where $f_*(\nu)$ denotes the pushforward of a measure $\nu$ by $f$. 
\end{theorem}

All the results for Pesin's theory were stated for $C^2$-diffeomorphisms, but they hold for $C^{1+\alpha}$-diffeomorphisms.

\subsection{The strong stable and strong unstable holonomies}

Let $f$ be a partially hyperbolic, dynamically coherent diffeomorphism. Each leaf of the foliation $\mathcal{F}^{cs}$ is foliated by strong stable manifolds. For a point $p\in M$ and $q\in W^{ss}_{1}(p)$, where $W^{ss}_{1}(p)$ is the strong stable manifold of size $1$, we can define the stable holonomy map restricted to the center-stable manifold, between center manifolds. Let us be more precise. We can choose two small numbers $R_1,R_2>0$, with the property that for any, $q\in W^{ss}_1(p)$, for any $z\in W^c_{R_1}(p)$, there is only one point in the intersection $W^{ss}_{2}(z) \cap W^c_{R_2}(q)$. We define $H^s_{p,q}(z) =W^{ss}_{2}(z) \cap W^c_{R_2}(q)$. With this construction we obtain a map $H^s_{p,q}: W^c_{R_1}(p) \to W^c_{R_2}(q)$. By the compactness of $M$ we can take the numbers $R_1$ and $R_2$ to be constants, independent of $p$ and $q$. 

We can define analogously the unstable holonomy map, for $p\in M$ and $q\in W^{uu}_1(p)$, which we will denote by $H^u_{p,q}: W^c_{R_1}(p) \to W^c_{R_2}(q)$.

In \cite{pughshubwilkinson97} and \cite{pughshubwilkinsoncorrection}, the authors prove that the map $H^s_{p,q}$ is $C^1$ if $f$ is a partially hyperbolic, center bunched and dynamically coherent $C^2$-diffeomorphism. Indeed, the authors prove that the strong stable foliation is $C^1$ when restricted to a center-stable leaf. Consider the family of $C^1$-maps $\displaystyle \{H^s_{p,q}\}_{p\in M, q\in W^{ss}_1(p)}$.
\begin{theorem}
\label{holonomies}
Let $f$ be an absolutely partially hyperbolic, dynamically coherent, $2$-normally hyperbolic and center bunched $C^2$-diffeomorphism. Suppose also that $\chi^c_-<1$ and $\chi^c_+>1$. Then the family $\{H^s_{p,q}\}_{p\in M, q\in W^{ss}_1(p)}$ is a family of $C^1$-maps depending continuously in the $C^1$-topology with the choices of the points $p$ and $q$. 
\end{theorem}

\begin{proof}
We follow the approach found in \cite{brownholonomy}, which is an approximation of the strong stable holonomies argument. In \cite{brownholonomy}, the author proves that such holonomies between center manifolds is $C^1$ if $f$ is $C^{1+\textrm{H\"older}}$ and verifies some stronger bunching condition, see section $2$ of \cite{brownholonomy} for precise statements. For a detailed proof in our setting we refer the reader to \cite{obataholonomies}.

Let $\pi^s_{.,.}$ be an approximation of the holonomy $H^s_{.,.}$. This means that there is a constant $C>0$, such that for any $p\in M$ and $q\in W^{ss}_1(p)$, there is a $C^2$-map, which is a diffeomorphism onto its image, $\pi^s_{p,q}: W^c_{R_1}(p) \to W^c(q)$ that verifies
\begin{enumerate}
\item $d(\pi^s_{p,q}(p),q) \leq C d(p,q)$;
\item $d(D\pi^s_{p,q}(p).v,v) \leq C d(p,q)$, where $v\in SE^c_p$ and $SE^c_p$ is the unit sphere on $E^c_p$;
\item if $p' \in W^c_{loc}(p)$ and $q' \in W^{ss}_1(p')  \cap W^c_{loc}(q)$, then $\pi_{p,q}^s$ coincides with $\pi^s_{p',q'}$ on $W^c_{loc}(p) \cap W^c_{loc}(p')$.
\end{enumerate}

This can be done in the following way: Consider a smooth subbundle $\widetilde{E}$ which is uniformly transverse to the subbundle $E^c$. Observe that the restriction of $\widetilde{E}$ to any center manifold is a $C^2$-bundle, since the center manifolds are $C^2$ by the $2$-normal hyperbolicity. For each point $q\in M$ and $\rho>0$, consider $L_{q,\rho}:=\exp_q(\widetilde{E}(q,\rho))$ to be the projection of the ball of radius $\rho$ by the exponential map over $q$. By the uniform transversality and the compactness of $M$, there exists a constant $\rho_0$ such that for any center leaf $W^c_{R_1}(p)$, the set $\{L_{q,\rho}\}_{q\in W^c_{R_1}(p)}$ forms an uniform foliated neighborhood of $W^c_{R_1}(p)$. Let $\pi^s_{p,q}$ be the holonomy defined by this local foliation, up to rescaling of the metric we may assume that it is well defined for $p\in M$ and $q\in W^{ss}_1(p)$. By the compactness of $M$ we obtain the constant $C>0$ above. Observe also that since the center leaves vary continuously in the $C^2$-topology, we obtain that the map $\pi^s_{p,q}$ varies continuously in the $C^2$-topology with the points $p$ and $q$.

For any $p,q\in M$ and each $n\in \N$, write $p_n = f^n(p)$ and $q_n = f^n(q)$. We define 
\[
H^s_{p,q,n} = f^{-n} \circ \pi^s_{p_n,q_n} \circ f^n.
\]

Since we are assuming that $f$ is absolutely partially hyperbolic, only for this proof, write its partially hyperbolic constants as $\chi_s = \chi^{ss}_+(p)$, $\chi_c= \chi^c_-(p)$ and $\widehat{\chi}_c = (\chi^c_+(p))^{-1}$. Also only for this proof, for a diffeomorphism $g:N_1 \to N_2$, between manifolds $N_1$ and $N_2$, we will write $g_*: SN_1 \to SN_2$, the action induced by the derivative on the unitary bundles of $N_1$ and $N_2$.

Observe that the Lipschitz norm of $f^{-1}_*$ restricted to a fiber $S_xE^c$ is $(\chi_c \widehat{\chi}_c)^{-1}$. Also since $f$ is a $C^2$-diffeomorphism, then $f^{-1}_*$ is a $C^1$-diffeomorphism of $SM$, let $C_1>0$ be the $C^1$-norm of $f^{-1}$ on $M$ and $C_2$ to be the $C^1$-norm of $f^{-1}_*$ on $SM$. For $\xi =(x,v) \in S_xM$, write $\xi_k =f^k_*(x,v)= (x_k,v_k)$, with $k\in \Z$.

In the setting that $f$ is $C^{1+ \textrm{H\"older}}$ and verifies a stronger bunching condition, Brown proves in \cite{brownholonomy} that $(H^s_{p,q,n})_{n\in \N}$ is a Cauchy sequence in the $C^1$-topology. Furthermore, this sequence converges exponentially fast to $H^s_{p,q}$. 

The stronger bunching condition is used to prove lemma $3.1$ in \cite{brownholonomy}. In our $C^2$ scenario, we can obtain a similar lemma, using that $\chi_c<1$ and $\widehat{\chi}_c <1$.

\begin{lemma}
\label{auxiliarylemma}
There are constants $\delta,\alpha\in (0,1)$ and $\theta \in (0,1)$, that verify the following: If $\xi=(x,v)$, $\zeta=(y,u) \in SW^c(p)$, $K>0$ and $n\geq 0$ verify $d(x_{n}, y_{n})< K \chi_s^n$, $d(\xi_n, \zeta_n) \leq K \chi_s^{n\theta}$ and for every $0\leq k \leq n$,
\[
d(x_k, y_k) \leq \delta. 
\]
Then, for all $0\leq k \leq n$,
\[
d(x_k,y_k) \leq K \chi_s^n.\chi_c^{-(n-k)} \textrm{ and } d(\xi_k, \zeta_k) \leq K \chi_s^{n\theta}.(\chi_c\widehat{\chi}_c)^{-(n-k)(1+\alpha)}.
\]
In particular, 
\[
d(\xi,\zeta) \leq K \chi_s^{n\theta}.(\chi_c \widehat{\chi}_c)^{-n(1+\alpha)}.
\]
Furthermore, $\theta$ and $\alpha$ can be chosen such that 
\[
\chi_s^{\theta}.(\widehat{\chi}_c\chi_c)^{-(1+\alpha)}<1.
\]
\end{lemma}

\begin{proof}
The proof is by backward induction in $k$. We will first denote by $\beta$, $\theta$, $\alpha$ and $\delta$ quantities that will be fixed later. Suppose that what we want holds for some $k \in \{1, \cdots n\}$, we will prove that it holds for $k-1$. Since $x_k$ and $y_k$ belongs to the same center manifold, it is easy to see that
\[
d(x_{k-1}, y_{k-1}) \leq \chi_c^{-1} d(x_k,y_k) \leq K\chi_s^n . \chi_c^{-n+k+1}.
\]

We have,

\[
\arraycolsep=1.2pt\def\arraystretch{2}
\begin{array}{rcl}
d(f^{-1}_*(x_k,v_k), f^{-1}_*(y_k,u_k)) & \leq & d(f^{-1}_*(x_k,v_k), f^{-1}_*(x_k,u_k)) + d(f^{-1}_*(x_k,u_k), f^{-1}_*(y_k,u_k))\\
& \leq & (\chi_c \widehat{\chi}_c)^{-1} d(v_k,u_k) + C_2 d(x_k,y_k).\\
& \leq & (\chi_c \widehat{\chi}_c)^{-1}[1 + C_2 .( \chi_c \widehat{\chi}_c) d(x_k,y_k)^{1-\beta}] . \max \{ d(x_k,y_k)^{\beta},d(v_k,u_k) \}\\
&\leq & (\chi_c \widehat{\chi}_c)^{-1}[1 + C_2 .( \chi_c \widehat{\chi}_c) \delta^{1-\beta}] \\
&& .K \max \{ \chi_s^{n\beta}.\chi_c^{-(n-k)\beta},  \chi_s^{n\theta}.(\chi_c\widehat{\chi}_c)^{-(n-k)(1+\alpha)}\},
\end{array}
\]
where the last inequality follows from our induction hypothesis.

We claim that we can choose $\alpha$, $\beta$ and $\theta$ such that for any $n\in \N$ and $0\leq k \leq n$ it holds
\[
\chi_s^{n\beta}.\chi_c^{-(n-k)\beta} \leq \chi_s^{n\theta}.(\chi_c\widehat{\chi}_c)^{-(n-k)(1+\alpha)}.
\]

This inequality is equivalent to 
\begin{equation}
\label{eq1}
1 \leq \chi_s^{n(\theta-\beta)}.(\chi^{(\beta-1-\alpha)}_c\widehat{\chi}^{-(1+\alpha)}_c)^{(n-k)}.
\end{equation}

Since $\widehat{\chi}^{-1}_c>1$, we can fix $\beta$ arbitrarily close to $1$ and $\alpha$ arbitrarily small such that
$1<\chi^{(\beta-1-\alpha)}_c\widehat{\chi}^{-(1+\alpha)}_c$. For the inequality above to hold we can just take any $\theta\in (0,\beta)$, so $\theta -\beta$ is negative. 

We also want that 
\begin{equation}
\label{eq2}
\chi_s^{\theta}.(\widehat{\chi}_c\chi_c)^{-(1+\alpha)}<1.
\end{equation}
By the center bunching condition, this holds if $\theta$ is close enough to $1$ and $\alpha$ is close enough to $0$. Fix $\beta\in (0,1)$ close to $1$, $\theta\in (0,\beta)$ close to $\beta$ and $\alpha>0$ small such that inequalities (\ref{eq1}) and (\ref{eq2}) hold.

Now take $\delta>0$ small enough such that
\[
[1 + C_2 .( \chi_c \widehat{\chi}_c) \delta^{1-\beta}]\leq (\chi_c \widehat{\chi}_c)^{-\alpha}.
\]

We conclude,
\[
\arraycolsep=1.2pt\def\arraystretch{2}
\begin{array}{rcl}
d(f^{-1}_*(\xi^k), f^{-1}_*(\zeta^k) &\leq &(\chi_c \widehat{\chi}_c)^{-(1+\alpha)}.K  \chi_s^{n\theta}.(\chi_c\widehat{\chi}_c)^{-(n-k)(1+\alpha)}\\
&=& K \chi_s^{n\theta}.(\chi_c\widehat{\chi}_c)^{-(n-k-1)(1+\alpha)}.
\end{array}
\]
\end{proof}

This lemma is specifically used to prove that the sequence $((H^s_{p,q,n})_*)_{n\in \N}$ is Cauchy. We can follow similar calculations as in \cite{brownholonomy} to conclude that for every $p\in M$ and $q\in W^{ss}_1(p)$ the sequence $(H^s_{p,q,n})_{n\in \N}$ is a Cauchy sequence that converges exponentially fast in the $C^1$-topology to $H^s_{p,q}$. The rate of convergence depends only on $\chi_s$, $\chi_c$ and $\widehat{\chi}_c$. In particular, it is independent on the choices of the points $p$ and $q$.

The family $\{\pi^s_{p,q}\}_{p\in M, q\in W^{ss}_1(p)}$ is a family of $C^2$-maps depending continuously in the $C^2$-topology with the choices of points $p$ and $q$. For each $n\in \N$, consider the family $\{f^{-n} \circ \pi^s_{p_n,q_n} \circ f^n\}_{p\in M, q\in W^{ss}_1(p)}$ and observe that, since $f$ is $C^2$, this is a family of $C^2$-maps depending continuously in the $C^2$-topology with the choices of the points $p$ and $q$.

Since the rate of convergence does not depend on the choices of the points $p$ and $q$, we conclude that the sequence of families $\left(\{f^{-n} \circ \pi^s_{p_n,q_n} \circ f^n\}_{p\in M, q\in W^{ss}_1(p)}\right)_{n\in \N}$ converges uniformly in the $C^1$-topology to the family $\{H^s_{p,q}\}_{p\in M, q\in W^{ss}_1(p)}$. Thus, the family $\{H^s_{p,q}\}_{p\in M, q\in W^{ss}_1(p)}$ is a family of $C^1$-maps depending continuously in the $C^1$-topology with the choices of $p$ and $q$. 
\end{proof}

\subsection{Berger-Carrasco's example}

Recall that for each $N\geq 0$ and $m=(x,y,z,w) \in \T^4$ we defined in section \ref{introduction} the diffeomorphism
\[
f_N(m) = (s_N(x,y) + P_x \circ A^N(z,w), A^{2N}(z,w)).
\] 

Observe that 
\[
Df_N(m) =  
\begin{pmatrix}
Ds_N(x,y)& P_x \circ A^N\\
0& A^{2N}
\end{pmatrix}.
\]
 
 It is useful to introduce $\Omega(x,y) = N \cos x + 2 $, so that 
 \[
 Ds_N(x,y) =
 \begin{pmatrix}
 \Omega(x,y) & -1\\
 1&0
 \end{pmatrix}.
 \]
For a point $m=(x,y,z,w)\in \T^4$, we will write $\Omega(m) = \Omega(x,y)$ and $Ds_N(m) = Ds_N(x,y)$. Observe that 
\begin{equation}
\label{basicineqfibered}
\frac{1}{2N} \leq \|Ds_N\| \leq 2N \textrm{ and } \|D^2s_N\| \leq N.
\end{equation}

 Let $A\in SL(2,\Z)$ be the linear Anosov matrix considered in the definition of the map $f_N$. Denote by $0<\lambda <1< \mu=\lambda^{-1}$ the eigenvalues of $A$. Let $e^s$ and $e^u$ be unit eigenvectors of $A$ for $\lambda$ and $\mu$, respectively.
 
Consider the involution $I(x,y,z,w) = (y,x,z,w)$ for $(x,y,z,w) \in \T^2$. An important feature of the map $f_N$ is given by the following lemma.
\begin{lemma}[\cite{bergercarrasco2014}, Lemma $1$]
\label{symmetrylemma}
The map $f_N^{-1}$ is conjugated to the map 
\[
(x,y,z,w) \mapsto (s_N(x,y) + P_x \circ A^{-N} (z,w), A^{-2N}(z,w)),
\]
by the involution $I$.
\end{lemma} 
 
 This lemma allows us to prove certain properties for $f_N$ and $f_N^{-1}$ only by considering the map $f_N$, since the involution tell us that $f_N$ and $f_N^{-1}$ behave in the same way up to exchange the $x$ and $y$ coordinates. This will be used several times throughout paper.
 
Recall that $E^c = \R^2 \times \{0\}$ and that the system $f_N$ is dynamically coherent.

\begin{proposition}
\label{coherence}
Fix $\varepsilon>0$ small, for $N$ large enough there is a $C^2$-neighborhood $\mathcal{U}_N$ of $f_N$, such that if $g\in \mathcal{U}_N$ then $g$ is dynamically coherent, its center leaves are $C^2$-submanifolds, $g$ is leaf conjugated to $f_N$ and for every $m\in \T^4$ the $C^2$-distance between $W^c_g(m)$ and $W^c_f(m)$ is smaller than $\varepsilon$.
\end{proposition}

\begin{proof}
Take $N$ large enough such that
\[
\lambda^{2N} < (2N)^{-4}.
\]
This inequality implies that $f_N$ is $2$-normally hyperbolic. Since its center foliation is smooth, by theorem $7.4$ of \cite{hps}, $f_N$ is plaque expansive. By theorem \ref{dynamicalcoherence}, for every $g$ sufficiently $C^2$-close to $f_N$, $g$ is dynamically coherent, leaf conjugated to $f_N$ and its center leaves are $C^2$-submanifolds.
Since the center foliation of $f_N$ is uniformly compact, from remark \ref{continuitycoherent}, if $\mathcal{U}_N$ is small enough then for every $g\in \mathcal{U}_N$ and $m\in \T^4$ the center leaves $W^c_g(m)$ and $W^c_f(m)$ are $\varepsilon$-close in the $C^2$-topology.  
\end{proof}

 Define $\pi_1(x,y,z,w)= (x,y) \in \T^2$ and $\pi_2(x,y,z,w) = (z,w)\in \T^2$. For convenience, a vector $(u,v) \in \R^2$ will be often identified with $(u,v,0,0) \in \R^4$, so that $Df_N(m).(u,v)= Df_N(m). (u,v,0,0)$.
For a vector $v\in T_m\T^4$ we will write $v_1 = D\pi_1(m).v$.

\section{ The size of the invariant manifolds and cone estimates}
\label{intmnfld}
In this section we obtain the main estimates to prove the ergodicity of $f_N$. Assuming proposition \ref{estimateprop} and fixing a small $\delta>0$, we prove:
\begin{proposition}
For $N$ large enough, for each ergodic component of the volume, for $f_N$, there exists a set with measure larger than $\frac{1-7\delta}{1+7\delta}$, such that:

For any $x$ in that set, there exist a stable and an unstable curves inside $W^c(x)$, with length bounded from below by $N^{-7}$. Moreover, the stable curve is transverse, inside $W^c(x)$, to the horizontal direction and the unstable curve is transverse to the vertical direction. 
\end{proposition} 
See lemma \ref{setzs} and proposition \ref{sizemnfld} for precise statements.

\begin{remark}
From now on the norm $\|.\|$ will be the norm induced by the usual metric of $\T^2$ or $\T^4$. We will omit the dependence of $N$ by writing $f= f_N$. 
\end{remark}

We fix two scales $\theta_1= N^{-\frac{2}{5}}$ and $\theta_2 = N^{-\frac{3}{5}}$.

\subsection{ Points with good contraction and expansion}
Since $f$ is non-uniformly hyperbolic, by theorem \ref{ergodiccomponentspesin}, there are at most countably many ergodic components. Therefore $Leb = \displaystyle \sum_{i\in \N} c_i \nu_i$, where $c_i \geq 0$ and for every $i\in \N$ the probability measure $\nu_i$ is $f$-invariant and ergodic. As a consequence of Birkhoff's theorem, for each measure $\nu_i$ there exists a set $\Lambda_i$ with full $\nu_i$-measure such that for every $m\in \Lambda_i$ 
\begin{equation}
\label{setlambda}
\displaystyle \frac{1}{n} \sum_{j=0}^{n-1} \delta_{f^j(m)} \xrightarrow[n\to + \infty]{} \nu_i \textrm{ and } \frac{1}{n} \sum_{j=0}^{n-1} \delta_{f^{-j}(m)} \xrightarrow[n \to +\infty]{} \nu_i \textrm{, in the $weak^*$-topology.}
\end{equation}
	
Where $\delta_p$ is the dirac mass on the point $p$. If $\nu_i\neq \nu_j$ then $\Lambda_i \cap \Lambda_j = \emptyset$. Define 
\begin{equation}
\label{lambda}
\Lambda = \displaystyle \bigcup _{i\in \N} \Lambda_i.
\end{equation}

Recall that $\mathcal{R}$ is the set of regular points given by Oseledets theorem. By remark \ref{osedetspesin}, the center direction is decomposed by the Oseledets splitting for almost every point, that is, for $m\in \mathcal{R}$ there is a decomposition $E^c_m = E^-_m \oplus E^+_m$, where $E^-_m$ is the Oseledets direction related to the negative center exponent and $E^+_m$ is the direction related to the positive exponent. 

For each $i\in \N$ define the sets 
\[
\arraycolsep=1.2pt\def\arraystretch{2}
\begin{array}{rcl}
Z_i^- & = & \left\{ m\in \mathcal{R} \cap \Lambda_i: \forall n\geq 0 \textrm{ it holds } \displaystyle \left \Vert Df^n(m)|_{E^{-}_m}\right \Vert < \left(N^{-\frac{4}{5}} \right)^n\right \};\\
Z_i^+ & = & \left \{ m\in \mathcal{R} \cap \Lambda_i: \forall n\geq 0 \textrm{ it holds }\displaystyle\left \Vert Df^{-n}(m)|_{E^{+}_m}\right \Vert < \left( N^{-\frac{4}{5}} \right)^n\right \};\\
Z_i & = & f(Z_i^-) \cap f^{-1}(Z_i^+).
\end{array}
\]

Define also 
\begin{equation}
\label{Z}
Z = \displaystyle \bigcup_{i\in \N} Z_i.
\end{equation}

\begin{remark}
\label{contaanterior}
For each $i\in \N$, by the definition of $Z_i$, $f^{-1}(Z_i) \subset Z_i^-$. Observe that 
\[
1\leq  \left \Vert Df(f^{-1}(m))|_{E^-_{f^{-1}(m)}}\right \Vert . \left \Vert Df^{-1}(m)|_{E^-_m}\right \Vert \leq N^{-\frac{4}{5}}\left \Vert Df^{-1}(m)|_{E^-_m}\right \Vert 
\]
We conclude that $\left \Vert Df^{-1}(m)|_{E^-_m}\right \Vert \geq N^{\frac{4}{5}}$. Similarly $\left \Vert Df(m)|_{E^+_m}\right \Vert \geq N^{\frac{4}{5}}$.
\end{remark}

We will need the following version of the Pliss lemma.

\begin{lemma}[ \cite{crovisierpujalsstronglydissipative}, Lemma $3.1$]
\label{pliss}
For any $\varepsilon>0$, $\alpha_1 < \alpha_2$ and any sequence $(a_i) \in (\alpha_1, +\infty)^{\N}$ satisfying
$$
\displaystyle \limsup_{n\to +\infty} \frac{a_0 + \cdots + a_{n-1}}{n} \leq \alpha_2,
$$
there exists a sequence of integers $0 \leq n_1 \leq n_2 \leq \cdots $ such that 
\begin{enumerate}
\item for any $k \geq 1$ and $n>n_k$, one has $\displaystyle \frac{a_{n_k} + \cdots + a_{n-1}}{(n-n_k)} \leq \alpha_2 + \varepsilon$;\\
\item the upper density $\displaystyle \limsup \frac{n_k}{k}$ is larger than $\displaystyle \frac{\varepsilon}{\alpha_2+ \varepsilon - \alpha_1}$.
\end{enumerate}
\end{lemma}

Using this lemma we prove the following.

\begin{lemma}
\label{setzs}
Fix $\delta>0$ small and assume that $N$ is large enough such that proposition \ref{estimateprop} holds for $f=f_N$. Then, it is verified $ \nu_i(Z_i) \geq \frac{1-7\delta}{1+7\delta}$ and $Leb(Z) \geq \frac{1-7\delta}{1+7\delta}.$

\end{lemma}

\begin{proof}
Since $N$ is large enough, by proposition \ref{estimateprop}, for every $m\in \mathcal{R}\cap \Lambda_i$, and since $E^-(m)$ is one dimensional, we obtain
\[
\displaystyle \lim_{n\to +\infty} \frac{1}{n} \log \|Df^n(m)|_{E^{-}_m}\|= \lim_{n\to +\infty}\frac{1}{n} \sum_{j=0}^{n-1} \log \|Df(f^j(m))|_{E^-_{f^j(m)}}\| \leq -(1-\delta) \log N.
\]

Take $\displaystyle \varepsilon = \frac{1}{6}\log N$, $\alpha_1 = -\log N - \log 2$, $\alpha_2 = -(1-\delta) \log N$ and consider the sequence $\left(\log \| Df(f^j(m))|_{E^-_m}\|\right)_{j\in \N}$. Applying Pliss lemma \ref{pliss} for those quantities we obtain a sequence of integers $(n_k)_{k\in \N}$ such that for every $k\in \N$ and $n> n_k$ 
\[
\displaystyle \frac{1}{n-n_k} \sum_{j=n_k}^{n-1} \log \|Df(f^j(m))|_{E^-_{f^j(m)}}\| \leq -(1-\delta) \log N + \frac{1}{6} \log N = \log N^{-\frac{5}{6}+ \delta} < \log N^{-\frac{4}{5}}.
\]

From this we conclude

\[
\|Df^n(f^{n_k}(m))|_{E^-_{f^{n_k}(m)}}\| < \left( N^{-\frac{4}{5}} \right)^n, \textrm{ } \forall n \geq 0.
\]

Thus for every $k\in \N$ we have $f^{n_k}(m) \in Z_i^-$. Since $m\in \Lambda_i$, by Birkhoff's theorem and the second point in Pliss lemma 

\[\arraycolsep=1.2pt\def\arraystretch{2}
\begin{array}{rcl}
\displaystyle \nu_i(Z_i^-) & \geq & \displaystyle \limsup_{k \to + \infty} \frac{n_k}{k} \\
&\geq &\displaystyle \frac{\varepsilon}{-(1-\delta) \log N + \varepsilon +\log N +\log 2}\\
&= & \displaystyle\frac{1}{ (1+6\delta) + \frac{6\log 2}{\log N}} \geq \frac{1}{1+7\delta}.
\end{array}
\]

Similarly, $\nu_i(Z_i^+) \geq \frac{1}{1+ 7\delta}$. This implies that
\[
\nu_i(\T^4-Z_i^*) \leq \frac{7\delta}{1+7\delta}, \textrm{ for $* = -,+ $}.
\]

By choosing $\delta>0$ small enough, the measure of these sets can be taken close to $1$. From the definition of $Z_i$ we conclude that

\begin{equation*}
\nu_i(Z_i) = 1- \nu_i(\T^4-Z_i) \geq 1- \frac{14\delta}{1+7\delta} = \frac{1-7\delta}{1+7\delta}.
\end{equation*} 

Since $Z= \displaystyle \bigcup_{i\in \N} Z_i$ and the previous estimate is valid for every $i\in \N$, then
\begin{equation*}
Leb(Z) \geq \frac{1-7\delta}{1+7\delta}.\qedhere
\end{equation*}
\end{proof} 
 
 Let $T=\left[ \frac{1+7\delta}{28\delta}\right]$, we may assume that $\delta>0$ is small enough such that $T>20$, define 
\begin{equation}
\label{X}
X=  \displaystyle \bigcap_{k=-T+1}^{T-1} f^k(Z).
\end{equation}

\begin{lemma}
\label{measure}
For $N$ large enough, if $\nu_i$ is an ergodic component of the Lebesgue measure then 
$$
\nu_i (X) >0.
$$

\end{lemma}

\begin{proof}
Recall that $\nu_i(Z_i)\geq \frac{1-7\delta}{1+7\delta} $, for $N$ large enough, this implies that 
$$
\nu_i(\T^4-Z_i) \leq \frac{14\delta}{1+7\delta}.
$$
Therefore
$$
\begin{array}{rcl}
\nu_i(X)  =  1- \nu_i(X^c)& \geq & 1- \displaystyle \sum_{j=-T+1}^{T-1} \nu_i(f^k(\T^4-H)) \\
&\geq & 1- \left(2\left[\frac{(1+7\delta)}{28\delta}\right]-2\right).\frac{14\delta}{1+7\delta} >0.
\end{array}
$$
\end{proof}

\subsection{Cone estimates}
\label{manysets}

Let $V\subset \R^2$ be a one dimensional vector subspace inside $\R^2$ and let $V^{\perp}$ be the one dimensional subspace perpendicular to $V$. For any vector $w\in \R^2$ we can write $w= w_V + w_{V^{\perp}}$, the decomposition of $w$ in $V$ and $V^{\perp}$ coordinates. For $\theta>0$ define 
\[
\C_{\theta}( V) = \{w\in \R^2: \theta \|w_V\| \geq \|w_{V^{\perp}}\|\},
\]
the cone inside $\R^2$ around $V$ of size $\theta$. For simplicity if $V = \R.(1,0)$ then we just write $\C^{hor}_{\theta}= \C_{\theta}(V)$ and  $\C_{\theta}^{ver} = \C_{\theta}(V^{\perp})$, we will call them the horizontal and vertical cones respectively. Throughout this paper, for a direction $V$, we will write 
\[
\C_{\theta}(V,m) = \C_{\theta}(V) \times \{0\} \subset T_m\T^4 = \R^2 \times \R^2.
\]

Recall that $\theta_1 = N^{-\frac{2}{5}}$. 

\begin{lemma}
\label{cone1}
For $N$ large enough, for every $ m\in Z$ we have that $E^+_m \subset \mathscr{C}_{\theta_1^{-1}}(m)$, with $\theta_1 = N^{-\frac{2}{5}}$. Furthermore, $\mathscr{C}_{\frac{\theta_1}{2}}(E^+_m,m) \subset \mathscr{C}^{hor}_{\frac{4}{\theta_1}}(m)$.The same is valid for the $E^-_m$ direction and the vertical cone.
\end{lemma}

\begin{proof}
From remark \ref{contaanterior}, we know that $\|Df(m)|_{E^+_m}\| \geq N^{\frac{4}{5}}$, for $m\in Z$. Take a vector of the form $(u,1)$, with $|u| \leq N^{-\frac{2}{5}}$, then for $N$ large enough
	\[\arraycolsep=1.2pt\def\arraystretch{1.5}
	\begin{array}{rclcl}
	\|Df(m).(u,1)\| &=& \|(u \Omega(m) -1, u)\|& \leq & |u||\Omega(m)| + 1 + |u|\\
	&\leq & |u|(N+2) + 1 + |u| & \leq& N^{-\frac{2}{5}}.N^{1 +\frac{1}{200}} + 1\\
	& \leq & N^{\frac{3}{5}+\frac{1}{200}}+1&\leq & N^{\frac{3}{5}+\frac{1}{100}} <N^{\frac{4}{5}}.
	\end{array}
	\]
		
	Hence, if $m\in Z$ then $E^+_m \subset \mathscr{C}_{\theta_1^{-1}}(m)$.
	
We want to determine $\theta>0$ such that the cone $\C^{hor}_{\theta}(m)$ contains the cone $\C_{\tilde{\eta}}(E^+_m,m)$. For this purpose we will consider a cone $\C_{\frac{\theta_1}{2}}(V,m)$, where the direction $V$ belongs to the boundary of the cone $\C^{hor}_{\theta_1^{-1}}(m)$.

Suppose $V$ is generated by the unit vector $(x,\frac{x}{\theta_1})$, with $x>0$. Observe that $V^{\perp}$ is generated by $(-\frac{x}{\theta_1},x)$. One of the boundaries of the cone $\C^{hor}_{\theta}(m)$ we are looking for is generated by the vector $\frac{\theta_1}{2}(-\frac{x}{\theta_1},x) + (x,\frac{x}{\theta_1})$.
	
	The size of the cone $\theta$ is given by 
	$$
	\theta = \frac{2.[x(\theta_1^2+2)]}{2x\theta_1} =\frac{\theta_1^2+2}{\theta_1} <\frac{4}{\theta_1}.
	$$
	
	Since the horizontal cones are symmetric with respect to the horizontal direction, we conclude that 
	\[
	\C_{\frac{\theta_1}{2}}(E^+_m,m) \subset\C^{hor}_{\theta}(m) \subsetneq \C^{hor}_{\frac{4}{\theta_1}}(m).
	\]
	
	By the symmetry of $f$, given by lemma \ref{symmetrylemma}, the same holds of the stable direction but using vertical cones.
\end{proof}

We define some critical regions. For that, define $I_1=I_1(N)= (-2N^{-\frac{3}{10}},2 N^{-\frac{3}{10}})$, $I_2=I_2(N) = \frac{I_1}{2}$, write $C_1 = \{\frac{\pi}{2}+I_1\} \cup \{\frac{3\pi}{2} +I_1\}$ and $C_2 =\{\frac{\pi}{2}+I_2\} \cup \{\frac{3\pi}{2} +I_2\}$. Consider the regions
\[\arraycolsep=1.2pt\def\arraystretch{1.5}
\begin{array}{rcl}
\displaystyle Crit_1 &=& \{ C_1 \times S^1  \times \T^2  \} \cup \{ S^1 \times C_1 \times \T^2 \}\\
\displaystyle Crit_2 &=& \{ C_2 \times S^1  \times \T^2  \} \cup \{ S^1 \times C_2 \times \T^2 \}.
\end{array}
\]

Write $G_* = (Crit_*)^c$, for $*=1,2$ and observe that $G_1 \subset G_2$. Observe also that each $G_*$ has four connected components, $\{G_{*,j}\}_{j=1}^4$. Each $G_{*,j}$ is a square and we can choose the index $j$ such that $G_{1,j} \subset G_{2,j}$.

\begin{remark}
\label{remarkado}
  The distance between the boundaries of these two sets is  
$$
d(\partial G_{1,j}, \partial G_{2,j}) = N^{-\frac{3}{10}} > N^{-7}, \textrm{ for $1 \leq j \leq 4$.}
$$
\end{remark}

Recall that $\theta_2= N^{-\frac{3}{5}}$.

\begin{lemma}
\label{esqueci}
If $N$ is large enough then
\begin{enumerate}
\item $Z \subset G_1 \subset G_2;$
\item If $m\in G_2$ then $Df(m).(\mathscr{C}^{hor}_{\frac{4}{\theta_1}}(m)) \subset \mathscr{C}^{hor}_{\theta_2}(f(m))$;
\item If $\gamma$ is a $C^1$-curve inside a center leaf, with length $l(\gamma)\geq N^{-\frac{3}{10}}$, such that $\gamma \subset G_2$ and is tangent to $\mathscr{C}^{hor}_{\theta_2}$ then $l(f(\gamma)) >4\pi$.
\end{enumerate}
Similar statements hold for the vertical cone and $f^{-1}$.
\end{lemma}

\begin{proof}

1. If $m\notin G_1$ then for $N$ large enough, $|\cos x| < 4N^{-\frac{3}{10}}$, in particular
\[
\|Df(m)|_{E^c_m}\| \leq N|\cos x| + 4 < 4N^{\frac{7}{10}-\frac{1}{200}} + 4 < N^{\frac{7}{10}-\frac{1}{100}} <N^{\frac{4}{5}},
\]
thus $Z \subset G_1 \subset G_2.$
\begin{enumerate}
\setcounter{enumi}{1}
\item For any $m\in G_2$, $(u,v) \in \mathscr{C}^{hor}_{\frac{4}{\theta_1}}(m)$ we have
\[
\theta_2 ( |\Omega(m)||u| - |v|) \geq \theta_2|u| \left( \frac{1}{2}.N^{\frac{7}{10}} - 2 - 4N^{\frac{3}{5}}\right)= |u|\left( \frac{1}{2}N^{\frac{1}{10}} - 2N^{-\frac{3}{5}} -4 \right) >|u|.
\]

\item For any $m\in G_2$ observe that  
\begin{equation}
\label{cosx}
|\cos x| \geq \frac{N^{-\frac{3}{10}}}{2}.
\end{equation}
For $(u,v) \in \mathscr{C}^{hor}_{\theta_2}(m)$ an unit vector, we must have
\[
\arraycolsep=1.2pt\def\arraystretch{1.5}
\begin{array}{rclcl}
\|Df(m).(u,v)\| & \geq & |\Omega(m)||u| - |v| & \geq & |u| ( |\Omega(m)| -\theta_2)\\
& \geq & \frac{\|(u,v)\|}{1+\theta_2} (|\Omega(m)| -\theta_2) & \geq &\frac{1}{2}( N|\cos x| -2 - \theta_2) \\
& \geq & \frac{N^{\frac{7}{10}}}{4} - 1- \frac{\theta_2}{2} &> &  N^{\frac{1}{2}}.
\end{array}
\]
Thus we have
\begin{equation*}
l(f(\gamma)) \geq N^{\frac{1}{2}}.N^{-\frac{3}{10}} = N^{\frac{2}{10}} >4\pi. \qedhere
\end{equation*}
\end{enumerate}
\end{proof}

 \begin{remark}
Observe that the condition $\gamma \subset G_2$ in the previous lemma can be replaced by $P_x(\pi_1(\gamma)) \subset P_x(\pi_1(G_2))$. The same holds for the past changing $P_x$ by $P_y$ and horizontal to vertical cones.
\end{remark}

\subsection{A lower bound on the size of the invariant manifolds}
\label{lowerbound}
Let $(S_n)_{n=0}^{+\infty}$ be a sequence of surfaces, such that each surface has a metric that induces a distance $d_n(.,.)$ and let $(\psi_n)_{n\in \N}$ be a sequence of diffeomorphisms $\psi_n:S_{n-1} \to S_{n}$. A curve $\gamma \subset S_0$ is a stable manifold for the sequence $(\psi_n)_{n\in \N}$ if any two points $x$ and $y$ on $\gamma$ verifies that $d_n(\psi_{n} \circ \cdots \circ \psi_1 (x), \psi_{n} \circ \cdots \circ \psi_1 (y))$ converges to zero exponentially fast. We say that $\gamma$ has size bounded from below by $r>0$, if $l_0(\gamma)\geq r$, where $l_0(.)$ is the length of $\gamma$ inside $S_0$.

The next proposition gives us the existence of stable and unstable curves tangent to the center direction, with good estimates on its sizes and its tangent directions. The proof of this proposition follows the exact same steps as theorem $5$ in \cite{crovisierpujalsstronglydissipative}, but with the changes necessary to get the estimates we need. 

Theorem $5$ in \cite{crovisierpujalsstronglydissipative} proves the existence of stable manifolds with uniform size and ``geometry" in the following scenario. Let $g:S \to S$ be a $C^2$-diffeomorphism of a compact surface and let $\sigma, \tilde{\sigma}, \rho, \tilde{\rho} \in (0,1)$ be constants such that 
\begin{equation}
\label{inequalityimp}
\frac{\tilde{\sigma} \tilde{\rho}}{\sigma \rho} > \sigma.
\end{equation}
For any point $x\in S$ having a direction $E \subset T_xS$ such that for all $n\geq 0$
\[
\tilde{\sigma}^n \leq \|Dg^n(x)|_{E}\| \leq \sigma^n \textrm{ and } \tilde{\rho}^n \leq \frac{\|Dg^n(x)|_E\|^2}{|\det Dg^n(x)|} \leq \rho^n.
\]  
They obtain stable manifolds for such points. Inequality (\ref{inequalityimp}) is important in the construction. That is why we need a good control on the Lyapunov exponent along the center, given by proposition \ref{estimateprop}.

\begin{proposition}
\label{sizemnfld}
For $N$ large enough, for each $m\in Z$, there are two $C^1$-curves $W^*(m)$ contained in $W^c(m)$, tangent to $E^*_m$ and with length bounded from below by $r_0 = N^{-7}$, for $*=-,+$. Those curves are $C^1$-stable and unstable manifolds for $f$, respectively. Moreover, $T_pW^+_{r_0}(m) \subset \mathscr{C}^{hor}_{\frac{4}{\theta_1}}(p)$ and  $T_qW^-_{r_0}(m) \subset \mathscr{C}^{ver}_{\frac{4}{\theta_1}}(q)$, for every $p\in W^+_{r_0}(m)$ and $q\in W^-_{r_0}(m)$.
\end{proposition}

\begin{proof}
We use some of the notation of the proof of Theorem 5 in \cite{crovisierpujalsstronglydissipative}. If $m\in Z$, by the definition of $Z$, $m\in Z_i$ for some $i\in \N$. Since $Z_i = f(Z_i^-) \cap f^{-1}(Z_i^+)$ we have that $f^{-1}(m)\in Z^-_i$, for this point it holds that
\[
(2N)^{-n}\leq \left \Vert Df^n(f^{-1}(m))|_{E^-_{f^{-1}(m)}}\right \Vert < \left(N^{-\frac{4}{5}}\right)^n, \textrm{ } \forall n\geq 0.
\]
Since $\left \vert \det Df(p)|_{E^c_{p}}\right \vert = \left \vert \det Ds_N(p)\right \vert = 1$ for every $p\in \T^4$, it also holds
\[
(2N)^{-2n}\leq \frac{\left \Vert Df^n(f^{-1}(m))|_{E^-_{f^{-1}(m)}}\right \Vert ^2}{\left|\det Df^n(f^{-1}(m))|_{E^c_{f^{-1}(m)}}\right|} <\left(N^{-2.\left(\frac{4}{5}\right)}\right)^n, \textrm{ } \forall n\geq 0.
\]

For each $n\in \N$ consider $\psi_n:V_{n-1} \to T_{f^{n}(m)}\T^2$ to be the lifted dynamics by the exponential map of the diffeomorphism $f|_{W^c(f^{n-1}(m))}$ along the orbit of $m$, that goes from a neighborhood $V_n$ of $0$ in $T_{f^{n-1}(m)}\T^2$ to a neighborhood of $0$ in $T_{f^{n}(m)}\T^2$. Since the center leaves are $C^2$, we have that $f|_{W^c(f^{n-1}(m))}$ is a $C^2$-diffeomorphism, this implies that $\psi_n$ is a $C^2$-diffeomorphism into its image.

Take $\sigma=N^{-\frac{4}{5}}$, $\tilde{\sigma} = (2N)^{-1}$, $\rho = \sigma^2$ and $\tilde{\rho} = \tilde{\sigma}^2$. Consider
$$
\lambda_1 = 2N^{-\frac{4}{5}} = 2\sigma \textrm{ and } \lambda_2 = \frac{1}{2.(2N)^2} = \frac{\tilde{\rho}}{2},
$$
and take
$$
C_0=3 > \displaystyle \sum_{k\geq 0 }\left(\frac{\sigma}{\lambda_1}\right)^k = 2 = \displaystyle \sum_{k\geq 0} \left(\frac{\lambda_2}{\tilde{\rho}}\right)^k.
$$

Let $E_n = E^-_{f^{n-1}(m)}$ and $F_n=E_n^{\perp}$ and use the basis $E_n \oplus F_n$. We define
$$
m_n = \left \Vert Df^n(m)|_{E^-_{f^{-1}(m))}}\right \Vert \textrm{ and } M_n = \frac{|\det Df^n|_{E^c}(f^{-1}(m))|}{m_n} = \frac{1}{m_n}.
$$ 
Using this notation it is also defined
\[\arraycolsep=1.2pt\def\arraystretch{2.5}
\begin{array}{c}
\displaystyle A_n= \sum_{k\geq 0} \lambda_1^{-k} m_{n+k}/m_n,\\
\displaystyle B_n = \sum_{k=0}^n \lambda_2^{k-n} \frac{M_k/M_n}{m_k/m_n}.
\end{array}
\]

The proof of theorem $5$ in \cite{crovisierpujalsstronglydissipative} gives
\begin{equation}
\label{anbn}
A_n \leq C_0 \left(\frac{\lambda_1}{\tilde{\sigma}}\right)^n \textrm{ and } B_n \leq C_0 \left(\frac{\rho}{\lambda_2}\right)^n.
\end{equation}

Define the change of coordinates in $T_{f^{n-1}(m)}\T^2$ given by $\Delta_n = Diag(A_n, A_n B_n)$, where the map $\Delta_n$ is defined using the coordinates $E_n \oplus F_n$. Observe that $A_n$ and $B_n$ are larger or equal to $1$, in particular, $\|\Delta_n\| = A_n B_n$ and $\|\Delta_n^{-1}\|=A_n^{-1}<1$.

Write $h_n= \Delta_{n+1} \circ \psi_n \circ \Delta_n^{-1}$ and $H_n = \Delta_{n+1} \circ D\psi_n(0) \circ \Delta_n^{-1}$. We have
$$
H_n = \begin{pmatrix}
a & d\\
0 & c
\end{pmatrix}
\textrm{ and } 
H_n^{-1} = \begin{pmatrix}
\frac{1}{a} & -\frac{d}{ca}\\
0 & \frac{1}{c}
\end{pmatrix}.
$$

From the proof of theorem $5$ in \cite{crovisierpujalsstronglydissipative}, we obtain

\begin{alignat}{3}
(\|Df|_{E^c}\|.\|Df^{-1}|_{E^c}\|^2)^{-1} & \leq & |a| & <&& \lambda_1                  \label{uno}\\
|a|\lambda_2^{-1} & \leq & |c| & \leq && \lambda_1 \lambda_2^{-1} \|Df|_{E^c}\|.\|Df^{-1}|_{E^c}\| + \lambda_1 \|Df^{-1}|_{E^c}\|^2\label{dos}\\
&&|d| &\leq && \|Df|_{E^c}\|.\|Df^{-1}|_{E^c}\||a|.\label{tres}
\end{alignat}

Using inequalities (\ref{dos}) and (\ref{tres}), we have 
$$
\left|\frac{d}{c}\right|\leq \frac{\|Df|_{E^c}\|.\|Df^{-1}|_{E^c}\| |a|}{|a| \lambda_2^{-1}} < \frac{(2N)^2}{2.(2N)^2} = \frac{1}{2}.
$$

Let us set $\xi= \frac{\tilde{\sigma} \lambda_2}{\lambda_1^2 \rho}$ and observe that for $N$ large enough $\xi>4$. For $\eta \leq \frac{1}{2}$ we will consider $\widetilde{\mathscr{C}}_{(\eta,n)}= \C_{\eta}(E_n)$ the cone of size $\eta$ around the direction $E_n$. If $(u,v)\in \widetilde{\mathscr{C}}_{(\eta,n+1)}$, using (\ref{uno}) and the estimate on $\left| \frac{d}{c}\right|$, we have
\[\arraycolsep=1.2pt\def\arraystretch{1.5}
\begin{array}{rclcl}
\|H_n^{-1}.(u,v)\| &\geq & \left|\frac{u}{a}\right|- \left|\frac{dv}{ca}\right|& \geq & \left|\frac{u}{a}\right| \left( 1-\left|\frac{d\eta}{c}\right| \right)  \\
& \geq & \frac{\|(u,v)\|}{(1+\eta) \lambda_1} \left( 1- \frac{\eta}{2}\right) & \geq & \frac{\|(u,v)\|}{\frac{3}{2} \lambda_1} . \frac{1}{2} .\frac{3}{2} = \frac{\|(u,v)\|}{2\lambda_1} > \frac{\|(u,v)\|}{\xi \lambda_1}.
\end{array}
\]

We conclude that the vectors of the cone $\widetilde{\mathscr{C}}_{(\eta,n+1)}$ are expanded by $\frac{1}{2 \lambda_1}$ by $H_n^{-1}$. Observe that if a linear map is $\frac{\eta}{6}$-close to $H_n^{-1}$ then the vectors inside $\widetilde{\C}_{\eta,n+1}$ are expanded by at least $(4\lambda_1)^{-1}> (\xi \lambda_1)^{-1}$. It is easy to see that such cone is contracted by any linear map $\frac{\eta}{6}$-close to $H_n^{-1}$.

Recall that $\|Df|_E^c\| \leq 2N$ and $\|D^2f^{-1}|_{W^c}\|\leq N$. Since $\|\Delta_{n+1}^{-1}\|<1$, we obtain
\[
\|Dh_n^{-1}(0)- Dh_n^{-1}(y)\|\leq \|\Delta_n\|.\|\Delta_{n+1}^{-1}\|.\|D^2f^{-1}|_{W^c}\|.\|\Delta_{n+1}^{-1}\|.\|y\| \leq N A_nB_n\|y\|.
\]
Using (\ref{anbn}), we have that $Dh^{-1}_n(y)$ is $\frac{\eta}{4|a|}$-close to $H_n^{-1}$ in a ball of radius
$$
\tilde{r}_{n+1} = \frac{\eta}{6N A_nB_n}> \frac{\eta}{6NC_0^2  } \left(\frac{\tilde{\sigma}\lambda_2}{\lambda_1 \rho}\right)^{n}>\frac{\eta}{54N}.(4\lambda_1)^{n}.
$$

Since $Dh_n^{-1}$ expands the vectors inside the cone $\widetilde{\C}_{\eta,n+1}$ by at least $(\xi \lambda)^{-1}>(4\lambda_1)^{-1}$, we can take 
\[
\tilde{r}_0 = \frac{\eta}{54N}.\frac{1}{4 \lambda_1} =\frac{\eta}{216 N \lambda_1}. 
\]

The proof of theorem $5$ in \cite{crovisierpujalsstronglydissipative} gives us a $C^1$-curve inside $T_{f^{-1}(m)}T^2$ tangent to the cone $\widetilde{\C}_{\eta,0}$, of size $\tilde{r}_0$, which is a stable manifold for the sequence $(h_n)_{n\in \N}$. 

To obtain a stable manifold for the sequence $(\psi_n)_{n\in \N}$ we need to apply $\Delta_0$ to this curve. Recall that $\Delta_0 = Diag(A_0,A_0)$, in particular it preserves the size and direction of a cone. Thus, we obtain that $\Delta_0(\widetilde{\C}_{(\eta,0)})= \C_{\eta}(E^-_{f^{-1}(m)})$.

To obtain a stable manifold for $f$, instead of the sequence $(\psi_n)_{n\in \N}$, we must project this curve by the exponential map, this projection will be denoted by $W^-(f^{-1}(m))$. Since $\T^2$ is the flat torus, the derivative of the exponential map is the identity. We conclude that the stable manifold for $f$ at the point $f^{-1}(m)$ is tangent to $\C_{\eta}(E^-_{f^{-1}(m)})$.

Now we estimate the size of the cones in the proposition at the point $m$. So far, the only restriction we have is $\eta \leq \frac{1}{2}$. Since $\|Df^{-1}|_{E^c}\|$ and $\|Df|_{E^c}\|$ are bounded from above by $2N$, 
$$
Df(f^{-1}(m)).\C_{\eta}(E^-_{f^{-1}(m)},f^{-1}(m)) \subset \C_{4N^2 \eta}(E^-_m,m).
$$

Using the estimates from lemma \ref{cone1}, we want $4N^2 \eta \leq \frac{\theta_1}{2} = \left(2N^{\frac{2}{5}}\right)^{-1}$, therefore, the additional restriction we put now is $\eta < \left(8N^{2+\frac{2}{5}}\right)^{-1}$. Since $N$ is large, we can take $\eta= N^{-3}$, for instance. By lemma \ref{cone1}, we have $E^-_m \subset \C^{ver}_{\theta_1^{-1}}$ and  $\C_{4N^2\eta}(E^-_m) \subset \C^{ver}_{\frac{4}{\theta_1}}$. This proves the estimate on the cones of the proposition. 

With this restriction, now we estimate the size of the stable manifold at the point $m$. For $\eta= N^{-3}$, we obtain for $N$ large enough, 
$$
\tilde{r}_0= \frac{2\eta}{216N\lambda_1} = \frac{1}{532.N^{4 - \frac{4}{5}}}>\frac{1}{N^5}.
$$

From this one can conclude that the stable manifold at the point $f^{-1}(m)$ has size bounded below by $N^{-5}$, this implies that at the point $m$ the stable manifold has size bounded by $(2N)^{-1}.N^{-5} > N^{-7} =r_0$, which concludes the proof for $W^-_{r_0}(m)$. The proof for the unstable manifold is analogous.
\end{proof}

\begin{remark}
From item $1$ of lemma \ref{esqueci} and Remark \ref{remarkado}, if $m\in Z$ then $W^*_{r_0}(m) \subset G_2$, for $*= -,+$.
\end{remark}
\section{Ergodicity of the system $f_N$}
\label{ergodicfiber}

In this section assuming proposition \ref{estimateprop}, we prove:

\begin{theorem}
For $N$ large enough $f_N$ is ergodic.
\end{theorem}

The proof is by contradiction. Suppose that $f=f_N$ is not ergodic, then there are at least two different ergodic components, $\nu_1$ and $\nu_2$. Let $\varphi:\T^4\to \R$ be a continuous function such that 
\[
\int \varphi d\nu_1 \neq \int \varphi d\nu_2.
\]

Consider the forward and backward Birkhoff's average 
$$
\varphi^+(m) = \displaystyle \lim_{n\to +\infty} \frac{1}{n}\sum_{j= 0}^{n-1}\varphi\circ f^j(m) \textrm{ and } \varphi^-(m) = \displaystyle \lim_{n\to +\infty} \frac{1}{n}\sum_{j= 0}^{n-1}\varphi\circ f^{-j}(m).
$$

Recall that we defined at the beginning of section \ref{intmnfld}, the set $\Lambda_i$ as the set of points such that any $m_i\in \Lambda_i$, it holds that $\varphi^+(m_i) = \varphi^-(m_i) = \int \varphi d\nu_i$, for $i=1,2$ and any continuous function $\varphi: \T^4 \to \R$.

First we remark that for almost every $ m\in \T^4$ the stable part of the Oseledets decomposition, defined in (\ref{oseledecsdirection}), is given by $E^s_m = E^{ss}_m \oplus E^-_m$. By theorem \ref{pesincontinuity} there is a $C^1$ stable Pesin manifold, $W^s(m)$, such that $T_mW^s(m) = E^{ss}_m \oplus E^-_m$, analogously for the unstable direction. Recall that the stable Pesin manifold has a topological characterization given by 
$$
W^s(m) = \{ y\in \T^4 : \displaystyle \limsup_{n\to +\infty} \frac{1}{n} \log d(f^n(m), f^n(y)) <0\}.
$$
The set $H$ was defined in (\ref{Z}). For $m\in Z$ consider 
\[
\displaystyle \widehat{W}^s(m) = \bigcup_{y\in W^-_{r_0}(m)}W^{ss}(y),
\]
where $W^-_{r_0}(m)$ is the stable manifold constructed in proposition \ref{sizemnfld} and $r_0 = N^{-7}$.

\begin{remark}
\label{c1sub}
By the topological characterization of the stable Pesin manifold we conclude that $\widehat{W}^s(m) \subset W^s(m)$. Observe that the strong stable manifold subfoliates the Pesin stable manifold, in particular $\widehat{W}^s(m)$ is open inside the Pesin manifold. We conclude that $\widehat{W}^s(m)$ is a $C^1$-submanifold and for every $m\in Z$ the stable and unstable Pesin manifolds contain a disc of size $r_0$. Analogously for the unstable manifold.
\end{remark}

Since $\varphi$ is continuous, for every $z\in W^s(m)$ and $w\in W^u(m)$, with $m\in \Lambda$, we obtain $\varphi^+(m) = \varphi^+(z)$ and $\varphi^-(w) = \varphi^-(m)$, where $\Lambda$ was defined in (\ref{lambda}) and has full Lebesgue measure.

\begin{claim}
There exists an invariant set $B$ of full Lebesgue measure, such that for every $m\in B$ and for Lebesgue almost every point $z \in W^u(m)$ it is verified $\varphi^-(z) = \varphi^+(z)$.
\end{claim}
\begin{proof}
If the claim does not hold, then there is a set of positive measure $C \subset \T^4$, such that for every $m\in C$ there is a set $C_m\subset W^u_{loc}(m)-\Lambda$, of positive Lebesgue measure inside $W^u_{loc}(m)$.

Observe that 
\[
\displaystyle \bigcup_{m\in C}C_m \subset \Lambda^c.\]

Suppose that $m\in C$ is a density point, take $T$ a small transversal to $W^u_{loc}(m)$ and consider $B(m,r)$ a small ball around $m$. By theorem \ref{abscontinuitystablepartition}, the unstable partition is absolutely continuous. In particular a Fubini-like formula holds. There is a set $Q$ of positive measure inside $T$, such that for every $q\in Q$, it holds that $W^u_{loc}(q) \cap C \cap B(m,r) \neq \emptyset$. Thus integrating the unstable measure of the sets $C_q \cap B(m,r)$ along $T$, we conclude that 
\[
Leb( \Lambda^c \cap B(m,r))>0.
\]  

This is a contradiction with the fact that $\Lambda$ has full measure.
\end{proof}

Recall that we defined $X= \displaystyle \bigcap_{k=-T+1}^{T-1}f^k(Z)$ and $\theta_2 = N^{-\frac{3}{5}}$. Recall also that we defined in section \ref{manysets} the sets $G_1$ and $G_2$.

\begin{lemma}
\label{bigmnfld}
For $N$ large enough and $n\geq 15$, for every $m\in X$ there are two curves $\gamma^-_{-n}(m) \subset f^{-n}(W^-_{r_0}(m))$ and $\gamma^+_n(m) \subset f^n(W^+_{r_0}(m))$ with length greater than $4\pi$. The tangent vectors of each of those curves are contained in the cone $\C^{ver}_{\theta_2}$ and $\C^{hor}_{\theta_2}$, respectively.
\end{lemma}

\begin{proof}

If $m\in X$ then 
\[
\{f^{-T+1}(m), \cdots , f^{T-1}(m)\}\subset Z \subset G_1\subset G_2, \textrm{ where $T=\left[ \frac{1+7\delta}{28\delta}\right]>20$.}
\]

 Define $W^+_k(m) = f^k(W^+_{r_0}(m))$ and observe that for every $z\in W^+_k(m)$, if $z\in G_2$ and $T_zW^+_k(m)\subset \C^{hor}_{\theta_2}$ then $T_{f(z)}W^+_{k+1}(m) \subset \C^{hor}_{\theta_2}$.

By Proposition \ref{sizemnfld}, $TW^+_0(m) \subset \C^{hor}_{\frac{4}{\theta_1}}$. Since $m\in Z \subset G_1$, by remark \ref{remarkado} we conclude that $W^+_0(m) \subset G_2$. Item $2$ of lemma \ref{esqueci} implies that $TW_1^+(m) \subset \C^{hor}_{\theta_2}$. 
 
If $p\in G_2$ and $(u,v) \in \C^{hor}_{\theta_2}(p)$ is an unit vector, then $\|Df(p).(u,v)\|> N^{\frac{1}{2}}$. For a $C^1$-curve $\gamma$ containing $m$ with length $N^{-7}$, such that $ \gamma \subset G_2$ and $T\gamma \subset \C^{hor}_{\theta_2}$, let $k\in \N$ be the largest number such that $f^j(\gamma) \subset G_2$, for every $j=1, \cdots ,k$. Since the vectors inside $\C^{hor}_{\theta_2}$ are expanded by at least $N^{\frac{1}{2}}$ and the cone $\C^{hor}_{\theta_2}$ is preserved by the derivative of the points in $G_2$, we conclude that $k\leq 14$.   

Let $k^+_0 \in \N$ be the smallest number such that $W^+_{k^+_0}(m) \cap \partial G_2 \neq \emptyset$. Recall that if $p\in G_2$ and $(u,v) \in \C^{hor}_{\frac{4}{\theta_1}}$, then by (\ref{cosx}), $\|Df(p).(u,v)\| >1$. Since $r_0 = N^{-7}$, we obtain that the curve $W_1^+(m) \subset \C^{hor}_{\theta_2}$ has length at least $N^{-7}$ and is tangent to $\C^{hor}_{\theta_2}$, by the previous paragraph $k^+_0 \leq 15$.

If $m\in X$, the connected component of $W^+_{k^+_0}(m)\cap G_2$ containing $f^{k^+_0}(m)$, which we will denote by $\widehat{W}^+_{k^+_0}(m)$, intersects the boundary of $G_2$ and $T\widehat{W}^+_{k^+_0}(m) \subset \C^{hor}_{\theta_2}$. Since $k_0^+ < T$, we know that $f^{k^+_0}(m) \in Z \subset G_1 \subset G_2$. We conclude that $\widehat{W}^+_{k^+_0}(m)$ also intersects the boundary of $G_1$.

 Let $\gamma^+_{k^+_0}$ be a connected component of $\widehat{W}^+_{k^+_0}(m) \cap (G_2-G_1)$, such that $\gamma^+_{k^+_0} \cap \partial G_1 \neq \emptyset$ and $\gamma^+_{k^+_0}\cap \partial G_2 \neq \emptyset$, see figure \ref{figure1}. The curve $\gamma^+_{k^+_0}$ is a $C^1$-curve that verifies the hypothesis of item $3$ from lemma \ref{esqueci}. Thus $l(f(\gamma^+_{k^+_0})) >4\pi$, $Tf(\gamma^+_{k^+_0}) \subset \C^{hor}_{\theta_2}$ and by definition $f(\gamma^+_{k^+_0}) \subset W^+_{k^+_0+1}(m)$. Define $\gamma_{k_0^+ + 1}(m) = f(\gamma^+_{k_0^+})$.
 
\begin{figure}[h]
\centering
\includegraphics[scale= 0.6]{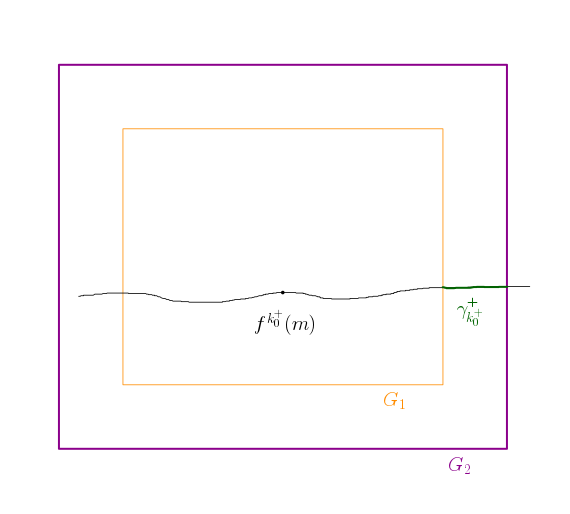}
\caption{The curve $\gamma_{k_0^+}^+$}
\label{figure1}
\end{figure}

Let 
\[\widetilde{G} = \left\{(x,y,z,w) \in \T^4: N^{-\frac{3}{10}} \leq |x - \frac{\pi}{2}| \leq 2N^{-\frac{3}{10}} \textrm{ or } N^{-\frac{3}{10}} \leq |x-\frac{3\pi}{2}| \leq 2N^{-\frac{3}{10}}\right\}.
\] 
It is easy to see that $\widetilde{G}$ has four connected components, each connected component having two boundaries. Since the critical region only depends on the coordinate $x$, for any point $p \in \widetilde{G}$, the derivative $Df(p)$ expands any vector inside $\C^{hor}_{\theta_2}$ by at least $N^{\frac{1}{2}}$.
 
We build $\gamma_n^+ \subset f(\gamma_{n-1}^+)$ inductively for $n> k_0^+ + 1$. Let us build it for $n= k_0^=+2$. Observe that $P_x(\pi_1(\gamma_{k^+_0+1}^+))= S^1$. Consider then $\widetilde \gamma^+_{k_0^++1}$ to be a connected component of $\gamma_{k_0^++1}^+(m) \cap \widetilde{G}$ that intersects the two boundaries of a connected component of $\widetilde{G}$. Define $\gamma^+_{k_0^++2}(m) = f(\tilde{\gamma}^+_{k_0^++1})$, observe that $l(\gamma^+_{k_0^++2}(m)) >4\pi$ and $Tf(\gamma_{k_0^++2}(m)) \subset \C^{hor}_{\theta_2}$. In this way we can build inductively the curves $\gamma_n^+(m)$ that satisfy the conclusions of the lemma. In a similar way we construct the curves $\gamma^-_{-n}(m)$. Since $k_0^+ \leq 15$ and $k_0^- \leq 15$, then this certainly holds for $n> 15$.
\end{proof}

For each $R>0$, $n\geq 15$ and $m\in X$, define 
\begin{equation}
\label{manifoldsat}
\displaystyle W^s_{R,-n}(m) = \bigcup_{q\in \gamma_{-n}^-(m)} W^{ss}_R(q),
\end{equation}
where the curve $\gamma_{-n}^-(m)$ is the curve given by lemma \ref{bigmnfld}. Define in a similar way the set $W^u_{R,n}(m)$. For the same reason as we explained in remark \ref{c1sub}, we obtain that $W^s_{R,-n}(m)$ and $W^u_{R,n}(m)$ are $C^1$-submanifolds.

\begin{lemma}
\label{transversality}
Fix $\theta_3>0$, such that $\theta_3> \theta_2$ and that satisfies $\C^{hor}_{\theta_3} \cap \C_{\theta_3}^{ver} = \{0\}$. There exists $0<R<1$ such that if $n\geq 15$, $m \in X$ and $m^- \in W^s_{R,-n}(m)$, then
\[
T(W^s_{2,-n}(m) \cap W^c(m^-)) \subset \C^{ver}_{\theta_3}.
\]
A similar result holds for $W^u_{R,n}(m)$. 

\end{lemma}

\begin{proof}
For any $p\in \T^4$, it holds that $\pi_2(W^{ss}(p)) = W^{ss}_A(\pi_2(p))$, where $W^{ss}_A(\pi_2(p))$ is the stable manifold of the point $\pi_2(p)$ for the linear Anosov system. Thus, given any point $q\in W_1^{ss}(p)$, for every $b\in W^c(p)$ there is only one point in $W^{ss}(b) \cap W^c(q)$. We define the stable holonomy map 
\[
\begin{array}{rlcl}
H^s_{p,q}:& W^c(p) & \longrightarrow & W^c(q)\\
& b & \mapsto & W^{ss}(b) \cap W^c(q).
\end{array}
\]
Locally this map is given by the holonomy map defined in section \ref{preliminary}. This is a $C^1$-diffeomorphism and we can naturally write $DH^s_{p,q}(p):\R^2 \to \R^2$.

From theorem \ref{holonomies} this family of maps vary continuously in the $C^1$-topology with the points $(p,q)$. Since $DH^s_{p,p}= Id$, by the compactness of $\T^4$, there is $R \in (0,1)$ such that for any $q\in W^{ss}_{R}(p)$ it holds $DH^s_{p,q}(p).(\C^{ver}_{\theta_2}) \subset \C^{ver}_{\theta_3}$. 

Observe that $W^s_{2,-n}(m)$ is contained inside a center-stable leaf, which is subfoliated by strong stable leaves. For this subfoliation, restricted to a center-stable leaf, the center manifolds are transversals. Thus for $m^- \in W^s_{R,-n}(m)$, the  $W^s_{2,-n}(m) \cap W^c(m^-)$ is given by $H^s_{m,m^-}(\gamma_{-n}^-(m))$. By our choice of $R$ and since $T\gamma_{-n}^-(m) \subset \C^{ver}_{\theta_2}$ the conclusion of the lemma follows.
\end{proof}

\begin{lemma}
\label{densecenter}
There is a set of full measure $D\subset \T^4$ such that for every $p\in D$ the orbit of $W^c(p)$ is dense among the center leaves.
\end{lemma}
\begin{proof}
For the linear Anosov $A^{2N}$, there is a set $D_A$ of full measure, with the property that every point in $D_A$ has dense orbit. This follows from the ergodicity of $A^{2N}$ for the Lebesgue measure. 

Since the Lebesgue measure of $\T^4$ is the product measure of the Lebesgue measure of each $\T^2$, take $D = \pi_2^{-1}(D_A)$. For any $p\in \T^4$ it holds that 
\[
\pi_2(f(W^c(p)) = A^{2N}(\pi_2(p)).
\] 

For any $q\in \T^2$, $\pi_2^{-1}(q)$ is a center leaf. Thus the dynamics among the center leaves is conjugated to $A^{2N}$ by $\pi_2$. Therefore, for any $p\in D$, since $\pi_2(p) \in D_A$ we conclude that the orbit of $W^c(p)$ is dense among the center leaves.
\end{proof}

Take $m_1\in  X\cap D \cap B \cap \Lambda_1$ and $m_2 \in X\cap D \cap B \cap  \Lambda_2$. From the definition of $\Lambda_1$ and $\Lambda_2$, for these two points
\[
\varphi^-(m_1) = \int \varphi d\nu_1 \textrm{ and } \varphi^+(m_2) =  \int \varphi d\nu_2.
\]

Fix a center leaf $W^c(q)$. Since $m_1,m_2 \in D$, there are two sequences $n_k \to +\infty$ and $l_j \to +\infty$, such that
\[
f^{n_k}(W^c(m_1)) \to W^c(q) \textrm{ and } f^{-l_j}(W^c(m_2)) \to W^c(q).
\]

By lemma \ref{bigmnfld}, there are curves $\gamma_{n_k}^+(m_1)$ and $\gamma_{-l_j}^-(m_2)$ with length bigger that $4\pi$ and contained in the cone $\C^{hor}_{\theta_2}$ and $\C^{ver}_{\theta_2}$, respectively. Take $R$ given by lemma \ref{transversality} and consider the sets

\begin{eqnarray*}
L^u_k(m_1) = \displaystyle \bigcup_{z\in \gamma^+_{n_k}(m_1)} W^{uu}_{R}(z)\subset W^u(f^{n_k}(m_1)) \\
L^s_j(m_2) = \displaystyle \bigcup_{z\in \gamma^-_{-l_j}(m_2)} W^{ss}_{R}(z) \subset W^s(f^{-l_j}(m_2)).
\end{eqnarray*}

For $k$ and $j$ large enough, $f^{n_k}(W^c(m_1))$ and $f^{-l_j}(W^c(m_2))$ are very close to the leaf $W^c(q)$. Thus by the control on the angles that we obtained in lemma \ref{transversality}, there is a transversal intersection between $L^u_k(m_1)$ and $ L^s_j(m_2)$. In particular $W^u(f^{n_k}(m_1))$ and $W^s(f^{-l_j}(m_2))$ intersects transversely. Before we continue with the proof we make the following remark.

\begin{remark}
\label{propertiesgeometrical}
This transverse intersection between stable and unstable manifolds is the key property to obtain ergodicity. We will see that the rest of the proof is a standard application of Hopf argument in the non-uniformly hyperbolic scenario. Three properties imply this transverse intersection:
\begin{enumerate}
\item For any point inside a certain set with positive measure for any ergodic component, there exists a stable curve inside the center manifold, with large size and controlled geometry. Similarly the existence of such a set but with the existence of an unstable curve. This is given by lemma \ref{bigmnfld}. Indeed, we can take the sets
\[
\displaystyle X^s= \bigcup_{n\geq 15} f^{-n}(X) \textrm{ and } X^u = \bigcup_{n\geq 15} f^n(X);
\]
\item The control of the holonomies, which will give a control on the tangent space of Pesin's manifolds considered in (\ref{manifoldsat}). This is given by lemma \ref{transversality};
\item The density of the orbit of almost every center leaf, which is given by lemma \ref{densecenter}. 
\end{enumerate}
\end{remark}

Now we continue with the proof. Fix $\varepsilon>0$ small and $l\in \N$ large enough such that the Pesin block $\mathcal{R}_{\varepsilon,l}$ has positive $\nu_2$ measure. By theorem \ref{pesincontinuity}, there is a number $\varepsilon_1>0$ such that every point $q\in \mathcal{R}_{\varepsilon,l}$ has a disc contained in $W^s(q)$ of size $\varepsilon_1$, which we will denote it by $W^s_{loc}(q)$. Furthermore, those discs vary $C^1$-continuously with the point $q\in \mathcal{R}_{\varepsilon.l}$.

Let $p$ be a point of transversal intersection between $L^u_k(m_1)$ and $ L^s_j(m_2)$. Take $M>0$ large enough such that $f^{M-l_j}(m_2) \in \mathcal{R}_{\varepsilon,l} $ and $d(f^{M-l_j}(m_2), f^M(p))< < \varepsilon_1$, such $M$ exists since $m_2$ is a typical point for $\nu_2$ and the set $R_{\varepsilon,l}$ has positive $\nu_2$-measure. We may assume that $f^{M-l_j}(m_2)$ is a density point of $\mathcal{R}_{\varepsilon,l} \cap \Lambda_2$. Fix a disc $T$ transverse to $W^s_{loc}(f^{M-l_j}(m_2))$ such that $\mathcal{R}_{\varepsilon,l} \cap \Lambda_2\cap T$ has positive measure inside $T$.

\begin{figure}[h]

\begin{center}
\begin{subfigure}[b]{0.4\linewidth}
\includegraphics[scale=0.4]{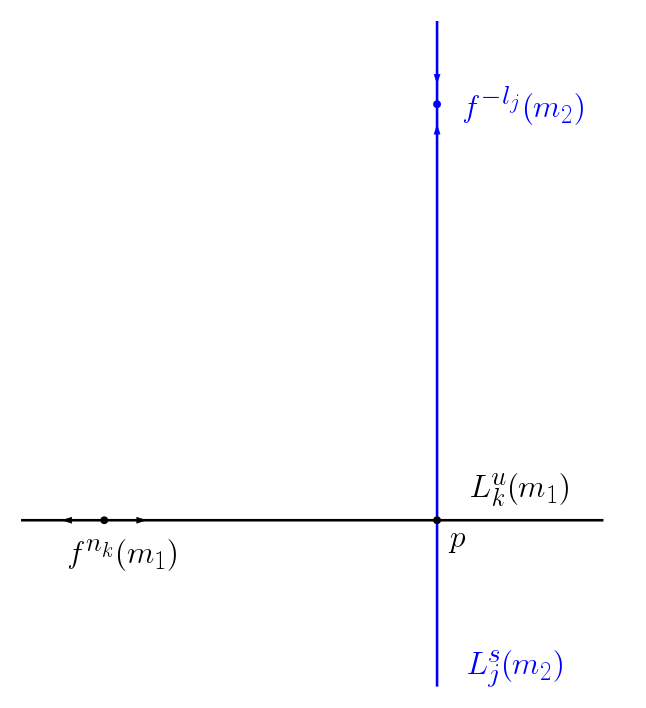}
\end{subfigure}
\begin{subfigure}[b]{0.4\linewidth}
\includegraphics[scale=0.4]{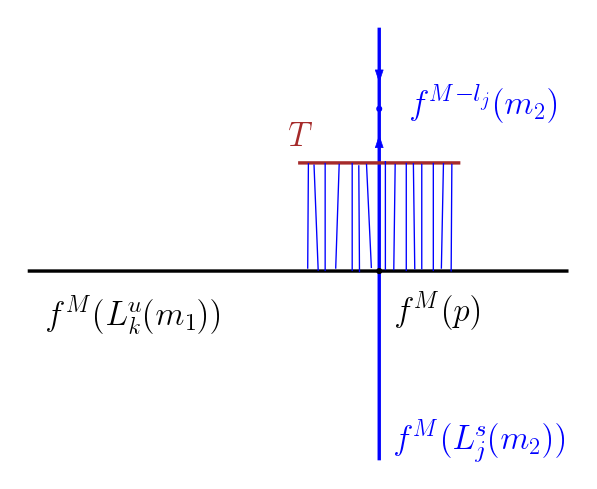}
\end{subfigure}
\caption{The transverse intersection and the holonomy}
\end{center}
\label{image2}
\end{figure}

Consider a disc $D^u \subset f^M(L^u_k(m_1))$ centered in $f^M(p)$ and observe that this disc is transverse to $W^s_{loc}(f^{M-l_j}(m_2))$. By the absolute continuity of the Pesin manifolds, we conclude that the set $A=\{W^s_{loc}(z) \cap D^u: z\in \mathcal{R}_{\varepsilon,l} \cap \Lambda_2\cap T\}$ has positive measure inside $W^u(f^{M+n_k}(m_1))$.

By the invariance of $B$, we know that $f^{M+n_k}(m_1)\in B$ and for almost every point $q\in W^u(f^{M+n_k}(m_1))$, it holds that $\varphi^+(q) = \varphi^-(q)$. Fix $\hat{z} \in A$ such that $\varphi^+(\hat{z}) = \varphi^-(\hat{z})$ and let $z\in \mathcal{R}_{\varepsilon,l} \cap \Lambda_2 \cap T$ be the point with $\hat{z} \in W^s_{loc}(z)$.

Since $z\in \Lambda_2$ and $\hat{z} \in W^s(z)$, we know that $\varphi^+(m_2) = \varphi^+(z)= \varphi^+(\hat{z})$. On the other hand, $\hat{z} \in W^u(f^{M+n_k}(m_1))$ implies that $\varphi^-(\hat{z}) = \varphi^{-}(m_1)$. Thus,

\[
\int \varphi d\nu_1 = \varphi^-(m_1) = \varphi^-(\hat{z}) = \varphi^+(\hat{z}) = \varphi^+(z) = \varphi^+(m_2) = \int \varphi d\nu_2.
\]
This is a contradiction since we assumed that $\int \varphi d\nu_1 \neq \int \varphi d\nu_2$. We conclude that there is only one ergodic component, in particular, the Lebesgue measure is ergodic. Thus we have proved that for $N$ large enough, $f_N = f$ is ergodic.

\section{Stable ergodicity of the system $f_N$}
\label{stableergodicity}

In this section we show how to adapt the proof of the ergodicity of $f_N$ to obtain $C^2$-stable ergodicity. Recall that for a vector $v\in T_m\T^4$, we defined $v_1= D\pi_1(m).v$ For a direction $E \subset T_m \T^4$ we will write $(E)_1 = D\pi_1(m). E$. For this section we fixed $0< \delta <<1$ small and we are assuming that $N$ is large and $\mathcal{U}_N$ is small enough such that proposition \ref{estimateprop} holds. Using proposition \ref{coherence} and the estimates in (\ref{basicineqfibered}), one easily obtains the following lemma.

\begin{lemma}
\label{manyconsiderations}
For each $\beta>0$, if $N$ is large and $\mathcal{U}_N$ is small enough, for $g\in \mathcal{U}_N$ it holds
\begin{enumerate}
	\item $g$ is partially hyperbolic, with a decomposition $TM = E^{ss}_g \oplus E^c_g \oplus E^{uu}_g$;
	\item $g$ is dynamically coherent and leaf conjugated to $f$ by a homeomorphism $h_g:\T^4 \to \T^4$;
	\item  $d_{C^2}(W_g^c(m),W^c_f(m)) \leq \beta$;
	\item $\|Dg(m)|_{E^c_{g,m}}\| \in ( e^{-\beta}\|Df(m)|_{E^c_{f,m}}\|, e^{\beta} \|Df(m)|_{E^c_{f,m}}\|);$
	\item $|\det Dg(m)|_{E^c_{g,m}}| \in (e^{-\beta},e^{\beta})$;
	\item $\|D^2g(m)|_{W^c_{g}(m)}\| \leq 2N$;
	\item  $ \max\{\|Dg(m)|_{E^c_{g,m}}\|, \|Dg^{-1}(m)|_{E^c_{g,m}}\|\} \leq 2N$;
	\item $ \min \{m(Dg(m)|_{E^c_{g,m}}), m(Dg^{-1} (m)|_{E^c_{g,m}})\}\geq (2N)^{-1}$;
	\item $\|Dg(m).v^c\| \in (e^{-\beta}\|Dg(m).v_1^c\|, e^{\beta} \|Dg(m).v_1^c\|)$, where $v^c \in E^c_{g,m}$;
	\item for points $p\in \T^4$ and $q\in W^c_{g}(p)$, let $exp^c_q: T_q W^c_g(p) \to W^c_g(p)$ be the exponential map of the center leaf. For any $C^1$-curve $\gamma \subset B(0,\frac{1}{2}) \subset T_q W^c_g(p)$, it holds $l_q ( \gamma) \in (e^{-\beta} l(exp^c_q(\gamma)), e^{\beta} l(exp^c_q(\gamma)))$, where $l_q(\gamma)$ is the length of the curve with respect to the inner product $<.,.>_q$ on $T_qW^c_g(p)$, the usual metric of $\T^4$ at the point $q$.	
\end{enumerate}
\end{lemma}

From now on we fix $0<\beta<<1$. By proposition \ref{estimateprop}, every diffeomorphism $g\in \mathcal{U}_N$ is non-uniformly hyperbolic. Using theorem \ref{ergodiccomponentspesin}, we obtain the ergodic decomposition $Leb = \displaystyle \sum_{i \in \N} c_i \nu_{g,i}$. We define similarly as in section \ref{intmnfld} the sets $\{\Lambda_{g,i}\}_{i\in \N}$. Let $\mathcal{R}_g$ be the set of regular points for $g$. For a regular point $p\in \mathcal{R}_g$, let $E^-_{g,p}$ and $E^+_{g,p}$ be the directions of the Oseledets splitting. It holds that $E^c_{g,p} = E^-_{g,p} \oplus E^+_{g,p}$.  

We define the sets

\[
\arraycolsep=1.2pt\def\arraystretch{2}
\begin{array}{rcl}
Z_{g,i}^- & = & \left \{ m\in \mathcal{R}_g \cap \Lambda_{g,i}: \forall n\geq 0 \textrm{ it holds } \displaystyle \left \Vert Dg^n(m)|_{E^{-}_{g,m}}\right \Vert  < \left(N^{-\frac{4}{5}} \right)^n\right \};\\
Z_{g,i}^+ & = & \left\{ m\in \mathcal{R}_g \cap \Lambda_{g,i}: \forall n\geq 0 \textrm{ it holds }\displaystyle\left \Vert Dg^{-n}(m)|_{E^{+}_{g,m}}\right \Vert < \left( N^{-\frac{4}{5}} \right)^n\right\};\\
Z_{g,i} & = & g(Z_{g,i}^-) \cap g^{-1}(Z_{g,i}^+);\\
Z_g &=& \displaystyle \bigcup_{i\in \N} Z_{g,i}.
\end{array}
\]

\begin{lemma}
\label{stablefirstlemma}
For every $g\in \mathcal{U}_N$, it holds that $ \nu_{g,i}(Z_{g,i}) \geq \frac{1-7\delta}{1+7\delta}$ and $Leb(Z_g) \geq \frac{1-7\delta}{1+7\delta}.$
\end{lemma}
The proof is analogous to the proof of lemma \ref{setzs}. Let $T=\left[ \frac{1+7\delta}{28\delta}\right]$ and define 
\begin{equation}
\label{Xg}
X_g=  \displaystyle \bigcap_{k=-T+1}^{T-1} g^k(Z_g).
\end{equation}
The proof of the next lemma is the same as the proof of lemma \ref{measure}.

\begin{lemma}
\label{measurerob}
For $N$ large and $\mathcal{U}_N$ small enough, if $\nu_{g,i}$ is an ergodic component of the Lebesgue measure then 
$$
\nu_{g,i} (X_g) >0.
$$

\end{lemma}

Now we make a few estimates on the cones. Recall that $\theta_1 = N^{-\frac{2}{5}}$.

\begin{lemma}
	\label{staberglemcone1}
	If $N$ is large and $\mathcal{U}_N$ is small enough then for each $g\in \mathcal{U}_N$, for every $m\in Z_g$, it is verified that $(E^+_{g,m})_1 \subset \C^{hor}_{\theta_1^{-1}}(m)$. Furthermore,  $\C_{\frac{\theta_1}{2}}((E^+_{g,g(m)})_1,m) \subset \C^{hor}_{\frac{4}{\theta_1}}(m)$. The same holds for the $E^-_{g,m}$ and the vertical cone.
\end{lemma}

\begin{proof}

For $m\in Z_g$, it holds that $\|Dg(m)|_{E^+_{g,m}}\| \geq N^{\frac{4}{5}}$. Take a vector of the form $(u,1)$, identifying $(u,1) = (u,1,0,0)$, with $|u| \leq N^{-\frac{2}{5}}$. For $N$ large enough and from the calculations made in the proof of lemma \ref{cone1}, which for this part does not use that $m\in Z_g$, we obtain 
\[\arraycolsep=1.2pt\def\arraystretch{1.5}
\begin{array}{rclcl}
\|Dg(m).(u,1)\| &\leq & e^{\beta}\|Df(m).(u,1)\| &\leq & e^{\beta}N^{\frac{3}{5}+\frac{1}{100}}  <N^{\frac{3}{5} + \frac{1}{50}}.
\end{array}
\]

Suppose that such $(u,1)$ generates $(E^+_{g,m})_1$, then
\[
\|Dg(m)|_{E^+_{g,m}}\| \leq e^{\beta} \frac{\|Dg(m).(u,1)\|}{\|(u,1)\|} \leq N^{\frac{3}{5} + \frac{1}{25}} < N^{\frac{4}{5}},
\]
which is a contradiction since $m\in Z_g$. The proof of the second part of the lemma is exactly the same as in lemma \ref{cone1}.
\end{proof}

Recall that we defined in section \ref{manysets} the sets $Crit_1$, $Crit_2$, $G_1$ and $G_2$. Also recall that $\theta_2 = N^{-\frac{3}{5}}$. We obtain the following lemma, by continuity and lemma \ref{esqueci}.

\begin{lemma}
\label{stabergestimates1}
For $N$ large, $\mathcal{U}_N$ small enough and $g\in \mathcal{U}_N$, it holds that
\begin{enumerate}
\item $Z_g \subset G_1 \subset G_2;$
\item If $m\in G_2$ then $\left(Dg(m).\mathscr{C}^{hor}_{\frac{4}{\theta_1}}(m)\right)_1 \subset \mathscr{C}^{hor}_{\theta_2}(g(m))$;\\
\item If $\gamma\subset G_2$ is a $C^1$-curve inside a center leaf such that the curve $\pi_1(\gamma)$ is tangent to $\C^{hor}_{\theta_2}$ and has length $l(\pi_1(\gamma))\geq N^{-\frac{3}{10}}$ then $l(g(\gamma)) > 4\pi$.
\end{enumerate}
Similar statements hold for the vertical cone and $g^{-1}$.
\end{lemma}
\begin{proof}
1. For $m\notin G_1$, by item $4$ of lemma \ref{manyconsiderations}, it holds 
\[
\|Dg(m)|_{E^c_{g,m}}\| \leq e^{\beta} \|Df(m)|_{E^c_{f,m}}\| < e^{\beta} N^{\frac{7}{10} - \frac{1}{100}} < N^{\frac{4}{5}}.
\]
\begin{enumerate}
\setcounter{enumi}{1}
\item The proof of item $2$ of lemma \ref{esqueci} actually gives that for $m\in G_2$, it holds
\[
Df(m).(\C^{hor}_{\frac{4}{\theta_1}}(m)) \subset \C^{hor}_{\frac{\theta_2}{K}}(f(m)),
\]
where $K = \frac{1}{2}N^{\frac{1}{10}} - 2N^{-\frac{3}{5}} -4 $. In particular, the inclusion of item $2$ of lemma \ref{esqueci} is uniformly strict. Thus, if $\mathcal{U}_N$ is small enough the conclusion follows.

\item From the estimates made in the proof of item $3$ of lemma \ref{esqueci} and by items $4$ and $9$ of lemma \ref{manyconsiderations}, it follows that 
\[
l(g(\gamma)) \geq l(g(\pi_1(\gamma))) > e^{-\beta}N^{\frac{1}{2}-\frac{3}{10}} > 4\pi.\qedhere 
\]
\end{enumerate}
\end{proof}

Now we estimate the size of the stable and unstable manifolds analogous to proposition \ref{sizemnfld}.

\begin{proposition}
	\label{stablesizenotfibered}
	Let $N$ be large and $\mathcal{U}_N$ be small enough. For $g\in \mathcal{U}_N$ and $m\in Z_g$, there are two $C^1$-curves, $W^*_g(m)$, contained in $W^c_g(m)$, tangent to $E^*_{g,m}$ and with length bounded from below by $r_0=N^{-7}$, for $*=-,+$. Those curves are $C^1$-stable and unstable manifolds for $g$, respectively. Moreover, $\left( T_pW^+_{g,r_0}(m) \right)_1\subset \mathscr{C}^{hor}_{\frac{4}{\theta_1}}(p)$ and  $\left(T_qW^-_{g,r_0}(m)\right)_1 \subset \mathscr{C}^{ver}_{\frac{4}{\theta_1}}(q)$, for every $p\in W^+_{g,r_0}(m)$ and $q\in W^-_{g,r_0}(m)$.
	 	
\end{proposition}

\begin{proof}
The main difference in the proof is that we have to project by $Dpi_1$ the tangent directions of the curves constructed. By lemma \ref{manyconsiderations} we will have good control of what happens after this projection, obtaining the desired estimates.
 
Using item $5$ of lemma \ref{manyconsiderations}, for $m\in Z_{g}$, it holds that 
\[
(2N)^{-n}\leq \left\Vert Dg^n(g^{-1}(m))|_{E^-_{g,g^{-1}(m)}}\right \Vert < \left(N^{-\frac{4}{5}}\right)^n,
\]
and
\[
(2N)^{-2n}e^{-n\beta}\leq \frac{\left \Vert Dg^n(g^{-1}(m))|_{E^-_{g,g^{-1}(m)}}\right \Vert ^2}{\left \vert \det Dg(g^{-1}(m))|_{E^c_{g,g^{-1}(m)}}\right \vert} <\left(e^{\beta}N^{-2.\left(\frac{4}{5}\right)}\right)^n.
\]

In the same way as in the proof of proposition \ref{sizemnfld}, consider the lifted dynamics $\psi_n:V_{n-1} \to T_{g^{n}(m)} W^c_g(g^{n}(m))$ of the diffeomorphism $g|_{W^c_g(g^{n-1}(m))}$, that goes from a neighborhood $V_n$ of $0$ in  $T_{g^{n-1}(m)} W^c_g(g^{n-1}(m))$ into a neighborhood of $0$ in $T_{g^{n}(m)} W^c_g(g^{n}(m))$. Since the center leaves are $C^2$, we have that $g|_{W^c_g(g^{n-1}(m)}$ is a $C^2$-diffeomorphism, which implies that $\psi_n$ is a $C^2$-diffeomorphisms into its image.

Take $\sigma = N^{-\frac{4}{5}}$, $\lambda_1= 2\sigma$, $\tilde{\sigma} = (2N)^{-1}$, $\rho= e^{\beta}\sigma^2$, $\tilde{\rho} = e^{-\beta}\tilde{\sigma}^2$, $\lambda_2 = \frac{\tilde{\rho}}{2}$ and $C_0=3$. Let $\xi =\frac{\tilde{\sigma}\lambda_2}{\lambda_1^2 \rho}$ and observe that for $N$ large enough 
\[
	\xi = \frac{\tilde{\sigma}\lambda_2}{\lambda_1^2 \rho}= 2^{-6} e^{-2\beta}N^{\frac{1}{5}} > 4.
\]
Following the same construction as in proposition \ref{sizemnfld}, one obtains the maps $\Delta_n$, $h_n$ and $H_n$. Recall that 
$$
H_n = \begin{pmatrix}
a & d\\
0 & c
\end{pmatrix}
\textrm{ and } 
H_n^{-1} = \begin{pmatrix}
\frac{1}{a} & -\frac{d}{ca}\\
0 & \frac{1}{c}
\end{pmatrix}.
$$

It also holds that 

\begin{alignat}{3}
(\|Dg|_{E_g^c}\|.\|Dg^{-1}|_{E_g^c}\|^2)^{-1} & \leq & |a| & <&& \lambda_1\label{runo}\\
|a|\lambda_2^{-1} & \leq & |c| & \leq && \lambda_1 \lambda_2^{-1} \|Dg|_{E_g^c}\|.\|Dg^{-1}|_{E_g^c}\| + \lambda_1 \|Dg^{-1}|_{E_g^c}\|^2\label{rdos}\\
&&|d| &\leq && \|Dg|_{E_g^c}\|.\|Dg^{-1}|_{E_g^c}\||a|.\label{rtres}
\end{alignat}

By item $4$ of lemma \ref{manyconsiderations} and using the previous inequalities
$$
\left|\frac{d}{c}\right|\leq \frac{\|Dg|_{E_g^c}\|.\|Dg^{-1}|_{E_g^c}\| |a|}{|a| \lambda_2^{-1}} < \frac{e^{2\beta}(2N)^2}{2e^{\beta}.(2N)^2} = \frac{e^{\beta}}{2}.
$$

For $\eta \leq \frac{1}{2}$ define the cone $\widetilde{\mathscr{C}}_{(\eta,n)}= \C_{\eta}(E_n)$, the cone of size $\eta$ around the direction $E_n$ inside $T_{g^{n-1}(m)}W^c_g(g^{n-1}(m))$. Using the estimate on $\left|\frac{d}{c}\right|$, following the same steps as in the proof of proposition \ref{sizemnfld}, we obtain that any linear map $\frac{\eta}{6}$-close to $H_n^{-1}$ contracts the cone $\widetilde{\C}_{(\eta,n+1)}$ and expands any vector inside $\widetilde{\C}_{(\eta,n+1)}$ by at least $\frac{1}{4\lambda_1}$.

By item $6$ of lemma \ref{manyconsiderations}, for any point $q\in \T^4$, it holds that $\|D^2g(q)|_{W^c_g(q)}\| \leq 2N$. Thus $(Dh_n(y))^{-1}$ is $\frac{\eta}{6}$-close to $H_n^{-1}$ in the ball of radius
\[
\tilde{r}_{n+1} = \frac{\eta}{12N\|\Delta_n\|} >\frac{\eta}{108N}(4\lambda_1)^n.
\]

Arguing similarly as in the proof of proposition \ref{sizemnfld}, we can take
\[
\tilde{r}_0 = \frac{\eta}{432 N \lambda_1}.
\]

Also by similar reasons as in the proof of proposition \ref{sizemnfld}, taking $\eta= N^{-3}$ we obtain a stable manifold for the sequence $(\psi_n)_{n\in \N}$ with size bounded from below by $\tilde{r}_0 > N^{-4+\frac{2}{5}}$, for $N$ large enough. The projection of this stable manifold by the exponential map gives the stable manifold $W^-_g(g^{-1}(m))$ for $g$ at the point $g^{-1}(m)$. By item $10$ of lemma \ref{manyconsiderations}, this stable manifold has size bounded from below by $e^{-\beta}. N^{-4+\frac{2}{5}}> N^{-5}$. Thus $W^-_g(m) = g(W^-_g(g^{-1}(m)))$ has size bounded from below by $r_0 = N^{-7}$.

The stable manifold for the sequence $(\psi_n)$ is tangent to the cone $\widetilde{\C}_{(\eta,0)}$ and at the origin is tangent to the direction $E^-_{g,m}$. By items $3$, $7$ and $8$ of lemma \ref{manyconsiderations}, for any $q\in \T^4$
\begin{equation}
\label{tamoainaatividade}
\left(Dg(q).(\widetilde{\C}_{2\eta,0})_1\right)_1 \subset \C_{e^{2\beta}8N^2 \eta}((E^-_{g,m})_1,m),
\end{equation}
where $(\widetilde{\C}_{2\eta,0})_1$ is identified with $(\widetilde{\C}_{2\eta,0})_1 \times \{0\}$.

 The stable manifold $W^-_g(g^{-1}(m))$ at the point $q$ is tangent to $Dexp^c_m((exp^c_m)^{-1}(q)). \widetilde{\C}_{(\eta,0)}$. If $\beta>0$ is small enough, then $Dexp^c_m(p)$ is close to be the identity, for any $p\in B(0,\frac{1}{2})$. Thus $\left( T_qW^-_g(g^{-1}(m))\right)_1 \subset \left( \widetilde{\C}_{2\eta,0}\right)_1$. By (\ref{tamoainaatividade}), we obtain
 \[
 \left( T_qW^-_{g,r_0}(m)\right)_1 \subset \C_{e^{2\beta}8N^2 \eta}((E^-_{g,m})_1,q).
 \]
 
By lemma \ref{staberglemcone1} and our choice of $\eta$, we conclude that 
\[
\left( T_qW^-_{g,r_0}(m)\right)_1 \subset \mathscr{C}^{ver}_{\frac{4}{\theta_1}}(q).\qedhere
\]
\end{proof}

So far we have obtained the results analogous to section \ref{intmnfld}. Now we will obtain the results analogous to the results used in section \ref{ergodicfiber} to obtain the ergodicity of $f$. The following is analogous to lemma \ref{bigmnfld}.

\begin{lemma}
\label{notfiberedbigmanifold}
For $N$ large, $\mathcal{U}_N$ small and $n>15$, if $\nu_{g,i}$ is an ergodic component of the Lebesgue measure, then for every $m\in X_g$ there are two curves $\gamma_{g,-n}^-(m) \subset g^{-n}(W^-_{g,r_0}(m))$ and $\gamma_{g,n}^+(m) \subset g^n(W^+_{g,r_0}(m))$ with length greater than $4\pi$, such that $\left( T \gamma^-_{g,-n}(m)\right)_1 \subset \C^{ver}_{\theta_2}$ and $\left( T\gamma^+_{g,n}(m)\right)_1 \subset \C^{hor}_{\theta_2}$.

\end{lemma}

\begin{proof}
The difference from the fibered case is to consider the projection by $\pi_1$. For $m\in X_g$, it holds that $W^+_{g,r_0}(m) \subset G_2$. Define $W^+_{k,g}(m) = g^k(W^+_{g,r_0}(m))$. By lemma \ref{stablesizenotfibered}, $\left( TW^+_{g,r_0}(m)\right)_1 \subset \C^{hor}_{\frac{4}{\theta_1}}$ and by lemma \ref{stabergestimates1}, $\left( TW^+_{1,g}(m)\right)_1 \subset \C^{hor}_{\theta_2}$. 

Construct in a similar way as in the proof of lemma \ref{bigmnfld} the number $k_0^+ \in \N$ and the curve $\gamma^+_{k_0^+,g}$. Since this curve must intersect $\partial G_1$ and $\partial G_2$, it has length $l(\pi_1 ( \gamma^+_{k_0^+,g})) \geq N^{-\frac{3}{10}}$ and $\pi_1(\gamma^+_{k_0^+,g})$ is tangent to $\C^{hor}_{\theta_2}$. By lemma \ref{stabergestimates1}, $l(g(\gamma^+_{k_0^+,g}))>4\pi$ and $\pi_1(g(\gamma^+_{k_0^+,g}))$ is tangent to $\C^{hor}_{\theta_2}$. The rest of the proof is the same as the proof of lemma \ref{bigmnfld}.
\end{proof}
For $R>0$, let
\[
\displaystyle W^s_{g,R,-n}(m) = \bigcup_{q\in \gamma_{g,-n}^-(m)} W_{g,R}^{ss}(q),
\]
where the curve $\gamma_{g,-n}^-(m)$ is the curve given by the previous lemma. Define similarly $W^u_{g,R,n}(m)$.  For the same reason as we explained in remark \ref{c1sub}, we obtain that $W^s_{g,R,-n}(m)$ and $W^u_{g,R,n}(m)$ are $C^1$-submanifolds. The next lemma is similar to lemma \ref{transversality}.
\begin{lemma}
\label{transversalnotfibered}
Fix $\theta_3>0$, such that $\theta_3 > \theta_2$ and satisfies $\C^{hor}_{\theta_3} \cap \C_{\theta_3}^{ver} = \{0\}$. For $g\in \mathcal{U}_N$, there exists $0<R<1$ such that if $n\geq 15$, $m \in X_g$ and $m^- \in W^s_{g,R,-n}(m) \subset W^s_{g,2,-n}(m)$, then
\[
\left( T(W^s_{g,2,-n}(m) \cap W_g^c(m^-))\right)_1 \subset \C^{ver}_{\theta_3}.
\]
A similar result holds for $W^u_{g,R,n}(m)$.
\end{lemma}

The main difference for the non fibered case is given in the following proposition.

\begin{proposition}
\label{brinarg}
For $N$ large and $\mathcal{U}_N$ small enough, if $g\in \mathcal{U}_N$ then for Lebesgue almost every point $m\in \T^4$ its central leaf $W^c_g(m)$ has dense orbit among the center leaves. 
\end{proposition}

\begin{proof}

For $\mathcal{U}_N$ small enough, for every $g\in \mathcal{U}_N$ there is a homeomorphism $h_g:\T^4 \to \T^4$, that takes center leaves of $f_N$ to center leaves of $g$, such that for every $m\in \T^4$ it is verified 
\[
g \circ h_g(W^c_f(m)) = h_g \circ f(W^c_f(m))
\]

Consider the quotients $M_f = \T^4/\sim^c_f $ and $M_g = \T^4/\sim^c_g$, where $p \sim^c_*q$ if and only if $q\in W^c_*(p)$ for $* = f, g$. We denote $\pi_f: \T^4 \to M_f$ and $\pi_g:\T^4 \to M_g$ the respective projections. Observe that $M_f = \T^2$ and that the induced dynamics $\tilde{f}:M_f \to M_f$ of $f$ is given by $A^{2N}$. Endow $M_g$ with the distance $d_g$ given by the Hausdorff distance on the center leaves, that is, 
\[
d_g(L,W) = d_{\textrm{Haus}}(\pi_g^{-1}(L), \pi_g^{-1}(W)).
\]

By the leaf conjugacy equation, the induced dynamics $\tilde{g}:M_g \to M_g$ of $g$ is conjugated to the linear Anosov $A^{2N}$ on $\T^2$ by the homeomorphism induced by $h_g$, which we will denote by $\tilde{h}_g$. Denote by $W^s_{A^{2N}}(.)$ the stable manifold of $A^{2N}$ on $\T^2$ and let 
\[
W^s_{\tilde{g}}(L)= \{W \in M_g: \displaystyle \lim _{n\to + \infty} d_g(\tilde{g}^n(L), \tilde{g}^n(W)) = 0\}, 
\] 
be the stable set of $L$. 

\begin{claim}

For every $m\in \T^4$, for every $q\in W^c_g(m)$, it is verified that 
\[
\pi_g(W^{ss}_g(q)) = W^s_{\tilde{g}}(\pi_g(m)) = \tilde{h}_g(W^s_{A^{2N}}(\pi_f(h_g^{-1}(m)))),
\]
and $\pi_g$ is a bijection from $W^{ss}_g(q)$ to $W^s_{\tilde{g}}(\pi_g(m))$.

\end{claim}

\begin{proof}

The leaf conjugacy equation implies that $ W^s_{\tilde{g}}(\pi_g(m)) = \tilde{h}_g(W^s_{A^{2N}}(\pi_f(h_g^{-1}(m))))$, in particular, $ W^s_{\tilde{g}}(\pi_g(m))$ is a continuous curve homeomorphic to a line.

It is immediate to see that $\pi_g(W^{ss}_g(q)) \subset W^s_{\tilde{g}}(\pi_g(m))$. We also have that $W^{ss}_g(q) \cap W^c_g(q) = \{q\}$. Indeed, since the angle between $E^c_g$ and $E^{ss}_g$ is uniformly bounded away from zero and the center foliation is uniformly compact, the map $\pi_g|_{W^{ss}_{g, loc}(z)}$ is injective, for every $z\in \T^4$ and for some small uniform size of stable leaf which we write $W^{ss}_{loc}(z)$. If there were two points $\{p, q \} \subset W^{ss}_g(q) \cap W^c_g(q) $ then for $n$ large enough $\{g^n(p), g^{n}(q)\} \subset W^{ss}_{g,loc}(g^n(q)) \cap W^c_g(g^n(q))$, which contradicts the fact that $\pi_g|_{W^{ss}_{g, loc}(q)}$ is injective. It remains to show the surjectivity.

We work inside $W^{cs}(m)$, which is foliated by strong stable manifolds. Take $P\in W^s_{\tilde{g}}(\pi_g(m))$ and consider its central leaf $F=\pi_g^{-1}(P)$. This is a transversal section of the $C^1$ foliation by strong stable manifolds inside the manifold $W^{cs}_g(m)$. Consider the set $L_{m,F} =\{z\in W^c_g(m): W^{ss}_g(z) \cap F \neq \emptyset\}$. 

Fix a small $\varepsilon >0$. Since the angle between $E^{ss}_g$ and $E^c$ is uniformly bounded away from zero and the center foliation is uniformly compact, for any point $p\in \T^4$, it holds that 
\[
\mathcal{V}^s_g(p) :=\displaystyle \bigcup_{q\in W^c_g(p)} W^{ss}_{g,\varepsilon}(q),
\]
contains a neighborhood of $W^c_g(p)$ inside $W^{cs}_g(p)$ of uniform size, independent of $p$.

Since $P\in W^s_{\tilde{g}}(\pi_g(m))$, take $n$ large enough such that $\pi^{-1}_g(\tilde{g}^n(P)) \cap\mathcal{V}^s_g(g^n(m)) \neq \emptyset $. Thus, there exists some $q_n\in  W^c_g(g^n(m))$ such that $W^{ss}_{g,\varepsilon}(q) \cap \pi^{-1}_g(\tilde{g}^n(P))\neq \emptyset$. We conclude that $W^{ss}_g(g^{-n}(q_n)) \cap F \neq \emptyset$, in particular, $L_{m,F} \neq \emptyset $.

If $\hat{p} \in L_{m,F}$ let $\gamma_{\hat{p},F}$ be a simple $C^1$ curve contained in $W^{ss}_g(\hat{p})$ connecting $\hat{p}$ and $F$, there is a foliated chart containing $\gamma_{\hat{p},F}$. Since $F$ is transversal to the foliation, we have that there is an open neighborhood of $\hat{p}$ inside $W^c_g(m)$ such that the strong stable manifold of every point in this neighborhood intersects $F$, thus $L_{m,F}$ is open.

Since $W^c_g(m)$ and $F$ are compact the distance, inside $W^{cs}_g(m)$, between them is smaller than a constant $\tilde{R}>0$. Observe that the tangent spaces of stable manifolds are contained inside a cone, transverse to the central direction in $W^{cs}_g(m)$. Thus, for $\hat{p} \in L_{m,F}$, the length of the piece of $W^{ss}_g(\hat{p})$ starting in $\hat{p}$ and ending in $F$ is bounded by a constant $C>0$. 

Let $(p_n)_{n\in \N} \subset L_{m,F}$ be a sequence such that $p_n \to p\in W^c_g(m)$. Consider $W^{ss}_{g,2C}(p)$ the strong stable manifold of size $2C$. Since compact parts of the strong stable manifold vary continuously with the point, $W^{ss}_{g,2C}(p_n)$ converges in the $C^2$-topology to $W^{ss}_{g,2C}(p)$. Take the sequence of points $(q_n)_{n\in \N}$ defined as $q_n \in W^{ss}_{g,2C}(p_n) \cap F$. Thus, $q_n \to q\in W^{ss}_{g,2C}(p)$ and since $F$ is closed, $q\in F$. Therefore $q\in W^{ss}_{g,2C}(p) \cap F$ and $L_{m,F}$ is closed. Since $W^c_g(m)$ is connected, it follows that $L_{m,F} = W^c_g(m)$.
\end{proof}

For the linear Anosov $A^{2N}$ the stable foliation is minimal. Let $m$ be a generic point of an ergodic component $\nu_{g,i}$ of the Lebesgue measure for $g$, suppose also that $m$ is a density point for the set $\Lambda_{g,i}$ defined at the  beginning of this section. By absolute continuity of the strong stable foliation almost every point inside $W^{ss}_g(q)$ is in the ergodic component of $m$, for $q\in \Lambda_{g,i}$. Using the minimality of the stable foliation of the linear Anosov and by the leaf conjugacy $W^s_{\tilde{g}}(\pi_g(m))$ is dense in $M_g$. 

Take $U$ a small open set in $M_g$. Since the center foliation is uniformly compact, $\hat{U} =\pi_g^{-1}(U)$ is a saturated open set such that any two center leaves in $\hat{U}$ are $C^2$-close to each other. By the previous claim $W^{ss}_g(m) \cap \hat{U} \neq \emptyset$.

Let $B(m,\varepsilon)$ be a small ball around $m$ such that $Leb(B(m,\varepsilon) \cap \Lambda_{g,i})$ has almost full measure inside $B(m,\varepsilon)$. By absolute continuity
\[
Leb(W^{ss}_g(B(m,\delta)\cap\Lambda_{g,i}) \cap\hat{U} \cap \Lambda_{g,i})>0.
\]
In particular $\nu_{g,i}(\Lambda_{g,i} \cap \hat{U})>0$. Since $m$ is a generic point for $\nu_{g,i}$, its future orbit visits $\hat{U}$ infinitely many times. This is true for any open set $U$ inside $M_g$, which concludes the proof of the proposition.
\end{proof}

Now let $N$ be large and $\mathcal{U}_N$ be small enough such that lemmas \ref{notfiberedbigmanifold}, \ref{transversalnotfibered} and proposition \ref{brinarg} hold. For $g\in \mathcal{U}_N$, if $g$ is not ergodic, we can follow the exact same steps as in the proof of ergodicity of $f$ and find a contradiction. We conclude that every $g\in \mathcal{U}_N$ is ergodic.

\section{The Bernoulli property}
\label{bernoullisection}
In this section we explain how to adapt the proof of ergodicity to obtain the Bernoulli property. Let $f= f_N$ for $N$ large enough. By theorem \ref{ergodiccomponentspesin}, since the Lebesgue measure is ergodic for $f$, there exists $k\in \N$ and probability measures $\nu_1, \cdots, \nu_k$, which are $f^k$-invariant, such that
\[
Leb = \displaystyle \frac{1}{k} \sum_{j=1}^k \nu_i,
\]
where each $(f^k,\nu_i)$ is Bernoulli. Suppose $k>1$. The measures $\{\nu_i\}_{i=1}^k$ form the ergodic decomposition of the Lebesgue measure for $f^k$. As we stated in remark \ref{propertiesgeometrical}, three properties imply the existence of transverse intersections between Pesin's manifolds of points in different ergodic components.

Observe that $f^{-k}(X^s) \subset X^s$, where we defined the set $X^s$ in item $1$ of remark \ref{propertiesgeometrical}. Similarly $f^k(X^u) \subset X^u$. Thus, item $1$ of remark \ref{propertiesgeometrical} is valid for $f^k$.

Once we have the curves obtained in item $1$ of remark \ref{propertiesgeometrical} and since a stable manifold for $f$ is a stable manifold for $f^k$, using the control on the holonomies given by lemma \ref{transversality} we obtain item $2$ of remark \ref{propertiesgeometrical}.

To obtain item $3$ of remark \ref{propertiesgeometrical} we need the following lemma.

\begin{lemma}
There is a set of full measure $D\subset \T^4$ such that for every $p\in D$ the $f^k$-orbit of $W^c(p)$ is dense among the center leaves.
\end{lemma}
\begin{proof}
The linear Anosov $A^{2N}$ is totally ergodic, that is, for any $j\in \N$, $A^{2Nj}$ is ergodic. In particular $A^{2Nk}$ is ergodic. The proof is the analogous to the proof of lemma \ref{densecenter}.
\end{proof}

Following the same steps of the proof of ergodicity for $f$, which is just Hopf argument in the non-uniformly hyperbolic scenario, we conclude that $f^k$ is ergodic. This is a contradiction, since the ergodic decomposition of the Lebesgue measure is given by the measures $\{\nu_{i}\}_{i=1}^k$ and $k>1$. Thus $k=1$. In particular $f$ is Bernoulli.

For $g\in \mathcal{U}_N$ to prove that $g$ is Bernoulli one follows the same steps as in the proof  that $f$ is Bernoulli. Observe that the stable and unstable foliations of $A^{2Nj}$ are minimal, for any $j\in \N$. With this observation one easily proves a lemma analogous to lemma \ref{brinarg}.

\section{Proof of proposition \ref{estimateprop}}
\label{exponents}

To prove proposition \ref{estimateprop}, we follow and adapt the proof of theorem \ref{bcexm} given by Berger and Carrasco in \cite{bergercarrasco2014} with the necessary changes. For a $C^1$-curve $\gamma$ and a measurable set $A\subset \gamma$, write $Leb(A)$ the measure of $A$ with respect to the Lebesgue measure in $\gamma$ induced by the metric of $\T^4$. Also denote $f= f_N$. In this section we will refer to the strong unstable manifold by unstable manifold.

\subsection{The estimate for $f_N$} 
The goal of this section is to prove the estimate given by proposition \ref{estimateprop} for $f$.

Recall that we denoted $e^u=(e^u_1, e^u_2)\in\R^2$ an unit eigenvector of $A$ for the eigenvalue $ 1< \mu = \lambda^{-1}$, where $\lambda\in (0,1)$ is the eigenvalue for the contractive direction of $A$. Recall also that we defined the linear map $P_x: \R^2 \to \R^2$ given by $P_x(a,b) = (a,0)$.

\begin{lemma}[\cite{bergercarrasco2014}, Proposition $1$]
\label{proposition1bc}
There is a differentiable function $\alpha: \T^4 \to \R^2$ such that the unstable direction of $f$ is generated by the vector field $(\alpha(m), e^u)$, where 
$$
Df(m).(\alpha(m), e^u)= \mu^{2N}(\alpha(f(m)), e^u) \textrm{ and } \|\alpha(m) - \lambda^N P_x(e^u)\| \leq \lambda^{2N}.
$$
\end{lemma}

\begin{definition}
\label{ucurve}
A $u$-curve is a $C^1$-curve $\gamma:[0,2\pi] \to M$ such that $\displaystyle\frac{d\gamma}{dt}(t) = \frac{(\alpha(\gamma(t)),e^u)}{\lambda^N\|P_x(e^u)\|} $, for every $t\in [0,2\pi]$.
\end{definition}

Observe that for a $u$-curve $\gamma$ 
\begin{equation}
\label{eqai}
\displaystyle\frac{df^k \circ \gamma}{dt}(t) = \frac{\mu^{2Nk}(\alpha(f^k(\gamma(t))),e^u)}{\lambda^N\|P_x(e^u)\|}, \textrm{ $\forall t\in [0,2\pi]$ and $\forall k\geq 0$.}
\end{equation}

The $u$-curves will play a fundamental role in the proof. The key property of a $u$-curve is that $\|\alpha(\gamma(t).(\lambda^N\|P_x(e^u)\|)^{-1}-(1,0)\| \leq \lambda^{2N}$. This will allow us to control the amount of time that a $u$-curve spend in a critical region, which is a region on $\T^4$ that only depends on the $x$ coordinate. 

Since we are interested in Lyapunov exponents along the center direction we will introduce certain types of vector fields along $u$-curves that will be useful in this task. After that we will be ready to give a criteria to obtain large positive Lyapunov exponents along the center direction for almost every point in $\T^4$.

\begin{definition}
\label{adaptedfield}
An adapted field $(\gamma,X)$ over a $u$-curve $\gamma$ is an unitary vector field $X$ such that
\begin{enumerate}
\item $X$ is tangent to the center direction;
\item $X$ is $(C_X,1/2)$-H\"older along $\gamma$, that is 
$$
\|X_m-X_{m^{\prime}}\| \leq C_X d_{\gamma}(m,m^{\prime})^{\frac{1}{2}},\textrm{ }\forall m, m' \in \gamma,
$$
where $C_X<20N^2\lambda^N$ and $d_{\gamma}$ is the distance measured along $\gamma$. 
\end{enumerate}
\end{definition}

Berger and Carrasco proved that for $N$ large enough and for every $(\gamma,X)$ adapted field 
$$\|X_m-X_{m^{\prime}}\| \leq \lambda^{N/3}, \textrm{ }\forall m,m' \in \gamma.$$

Fix an adapted field $X$ and denote by $X^k= \frac{(f^k)_*X}{\|(f^k)_*X\|}$, where 
\[
\left((f^k)_*X\right)_m = Df^k(f^{-k}(m)).X_{f^{-k}(m)}.
\]

\begin{lemma}[\cite{bergercarrasco2014}, Lemma $2$]
\label{bclemma2}
For $N$ large enough, for every adapted field $(\gamma,X)$, for every $k\geq 0$ and  every $1\leq j \leq [\mu^{2Nk}]$, the pair $(\gamma^k_j, X^k|_{\gamma^k_j})$ is an adapted field.
\end{lemma}

 Denote by $d\gamma$ the Lebesgue measure induced on $\gamma$ and by $|\gamma|$ the length of $\gamma$. Define
 \[
 I_n^{\gamma,X} \coloneqq  \displaystyle \frac{1}{|\gamma|}\int_{\gamma} \log \|Df^n.X\| d\gamma.
 \]
 
 Now we prove the following criteria to obtain positive Lyapunov exponents along the center direction.

\begin{proposition}
\label{exp1}
Suppose that there exists $C>0$ such that for every $u$-curve $\gamma$ there is an adapted vector field $X$ which satisfies for $n$ large enough 
$$
\frac{I^{\gamma,X}_n}{n}>C.
$$
Then Lebesgue almost every point in $\T^4$ has a Lyapunov exponent along the central direction which is larger than $(1-2\lambda^{2N}) C$. 
\end{proposition}
\begin{proof}
We will prove that for every $\rho>0$, for almost every point there is a Lyapunov exponent greater than $(1-2\lambda^{2N}-\rho)C$ in the center direction.
Suppose not, then there is a set with positive measure $B$ such that every point in this set does not have a Lyapunov exponent greater than $(1-2\lambda^{2N}-\rho) C$. Since the unstable foliation is absolutely continous there is an unstable manifold $L^u$ that intersects $B$ in a subset with positive Lebesgue measure inside $L^u$. Let $q\in L^u$ be a Lebesgue density point of $L^u\cap B$.

Let $r_k =2\pi \lambda^{2Nk}$ and let $\gamma_{r_k}:[-r_k,r_k] \to M$ to be a piece of $u$-curve such that $\gamma_{r_k}(0)= q$, since $q$ is a density point then
$$
\displaystyle \frac{Leb(\gamma_{r_k} \cap B)}{Leb(\gamma_{r_k})} \to 1.
$$

Take $\beta <\rho$ and let $k$ be large enough such that $Leb(\gamma_{r_k} \cap B^c) < \frac{\beta C}{\log2N} Leb(\gamma_{r_k})$. Observe that $f^k \circ \gamma_{r_k}$ is a $u$-curve, let $X_{r_k}$ be the vector field over $\gamma_{r_k}$, such that $(f^k \circ \gamma_{r_k}, (f^k)_*X_{r_k})$ satisfies the hypothesis of the lemma. Let 
\[
\chi(m) = \displaystyle \limsup_{n\to \infty} \frac{\log \|Df^n(m).X_{r_k}\|}{n},\]
thus for every $m\in B$, $\chi(m)<(1-2\lambda^{2N}-\rho) C$. From (\ref{eqai}) and lemma \ref{proposition1bc}, for $N$ large enough we obtain
\[
\displaystyle \frac{1}{\left\|\frac{d(f^k \circ \gamma_{r_k})}{dt}\right\|} \geq \frac{1-2\lambda^{2N}}{\mu^{2Nk}}.
\]

In particular,

\[\arraycolsep=1.2pt\def\arraystretch{1.5}
\begin{array}{rcl}
\displaystyle \int_{\gamma_{r_k}} \chi d\gamma_{r_k} &=& \displaystyle \int_{f^k \circ \gamma_{r_k}} \chi \circ f^{-k} \frac{1}{\left\|\frac{d(f^k \circ \gamma_{r_k})}{dt}\right\|} d(f^k \circ \gamma_{r_k})\\
&\geq&\displaystyle \frac{1- 2\lambda^{2N}}{\mu^{2Nk}} \int_{f^k\circ \gamma_{r_k}} \chi\circ f^{-k} d(f^k \circ \gamma_{r_k})\\
& = & \displaystyle\lambda^{2Nk}(1-2\lambda^{2N})\limsup_{n\to + \infty} \int_{f^k\circ \gamma_{r_k}} \frac{\log\|Df^n(m).(f^k)_*X_{r_k}\|}{n} d(f^k \circ \gamma_{r_k})\\
&\geq & \lambda^{2Nk}(1-2\lambda^{2N})|f^k \circ \gamma_{r_k}| C >(1-2\lambda^{2N})C|\gamma_{r_k}|. 
\end{array}
\]

On the other hand.

\[\arraycolsep=1.2pt\def\arraystretch{1.5}
\begin{array}{rcl}
\displaystyle \int_{\gamma_{r_k}} \chi d\gamma_{r_k} & = & \displaystyle \int_{\gamma_{r_k}\cap B} \chi d\gamma_{r_k}+ \int_{\gamma_{r_k}\cap B^c} \chi d\gamma_{r_k}\\
&\leq &(1-2\lambda^{2N}-\rho) C |\gamma_{r_k}| + \log2N.C.\beta(\log2N)^{-1} |\gamma_{r_k}|\\
&=& (1-2\lambda^{2N}-\rho +\beta)C|\gamma_{r_k}| <(1-2\lambda^{2N})C|\gamma_{r_k}|. 
\end{array}
\]

Which is a contradiction. Since it holds for every $\rho>0$, one concludes the proof of the proposition.
\end{proof}

We can represent the curve $f^k \circ \gamma$ as the concatenation $f^k \circ \gamma = \gamma_1^k \ast\cdots \ast \gamma_{[\mu^{2Nk}]}^k \ast \gamma_{[\mu^{2Nk}]+1}^k$, where $\gamma^k_i$ is a $u$-curve for every $1\leq i \leq [\mu^{2Nk}]$, $\gamma^k_{[\mu^{2Nk}]+1}$ is a piece of a $u$-curve, $[.]$ denotes the integer part of a number and $\ast$ denotes the concatenation between the curves. Berger and Carrasco proved the following formula, see section $3$ of \cite{bergercarrasco2014}.
\begin{lemma}
\label{formula1}
For every adapted field $(\gamma,X)$ and $n\in \N$, for each $k=0,\cdots, n-1$ there exists a number $\beta_k \in [-2 \lambda^{2N}, 2\lambda^{2N}]$ such that 
\[\arraycolsep=0.8pt\def\arraystretch{1.5}
\begin{array}{rcl}
I_n^{\gamma,X} &= & \displaystyle \frac{1}{|\gamma|}\int_{\gamma} \log \|Df^n.X\| d\gamma\\
&=& \displaystyle  \sum_{k=0}^{n-1}\frac{1+\beta_k}{\mu^{2Nk}|\gamma|} \left(\sum_{j=1}^{[\mu^{2Nk}]}\int_{\gamma^k_j} \log\|Df.X^k\| d\gamma_j^k + \int_{\gamma^k_{[\mu^{2Nk}]+1}} \log\|Df.X^k\| d\gamma_{[\mu^{2Nk}]+1} \right),
\end{array}
\]
where $\beta_k \in [-2\lambda^{2N},2\lambda^{2N}]$.
\end{lemma}

This formula will allow us to study the growth of $I_n^{\gamma,X}$ by studying the pieces $\int_{\gamma^k_j} \log\|Df.X^k\| d\gamma_j^k$. In order to analyze these pieces we will define the notion of ``good" and ``bad" pieces. The estimate on the growth of $I_n^{\gamma,X}$ will come from an induction on $n$ and a combinatorial argument, to estimate the number of ``good" and ``bad" pieces that appears in this formula. 

Fix $\tilde{\delta}>0$ small, the number $N$ will be chosen after in function of $\tilde{\delta}$. Let 
$$
E(\gamma,X) = I^{\gamma, X}_1=  \frac{1}{|\gamma|} \int_{\gamma} \log \|Df(m).X_m\| d\gamma.
$$

Recall that for $m=(x,y,z,w)\in \T^4$, we defined $\Omega(m)= N\cos(x) + 2$. Define $v_m= (1,\Omega(m))$ and $u_m = (\Omega(m), -1)$. They form an orthogonal basis of the center direction. Let $X$ be an unit vector field tangent to the center direction, thus using this basis we have
$$
X_m = \frac{\cos(\theta_X(m))}{\sqrt{1+\Omega(m)^2}}v_m + \frac{\sin(\theta_X(m))}{\sqrt{1+\Omega(m)^2}}u_m.
$$

Where $\theta_X(m)$ is the angle that $X_m$ makes with $v_m$. Using the basis $(v_m, u_m)$ the derivative can be written as

\[
Df(m).X_m = \left(\sin(\theta_X(m)). \sqrt{1+\Omega(m)^2}, \frac{\cos(\theta_X(m)) + \sin(\theta_X(m)). \Omega(m)}{\sqrt{1+(\Omega(m))^2}}\right),
\]

then 
\[
\|Df(m).X_m\| \geq\left|\sin(\theta_X(m))\right|. \sqrt{1+\Omega(m)^2}\geq  |\sin(\theta_X(m))|.|\Omega(m)|.
\]

If $N$ is large enough and if $|x-\pi/2| > 2.N^{-\tilde{\delta}}$ and $|x-3\pi/2| \geq 2.N^{-\tilde{\delta}}$ then $|\cos(x)|\geq N^{-\tilde{\delta}}$.

Define the critical strip as 
\[
Crit = \left\{(x,y,z,w) \in \Tor^4: |x-\pi/2|< 2.N^{-\tilde{\delta}} \textrm{ or } |x-3\pi/2| < 2.N^{-\tilde{\delta}}\right\},
\]
thus the length of the projection of the critical strip on the first coordinate is $l(Crit) < 8.N^{-\tilde{\delta}}$, which converges to zero as $N$ goes to infitiny.

\begin{lemma}
\label{lemmaanterior}
For $N$ large enough, if $m\notin Crit$ then $|\Omega(m)| \geq N^{1-2\tilde{\delta}}$ and $\|Df(m). X_m\| \geq N^{1-2\tilde{\delta}}.|\sin(\theta_X(m))|.$ 
\end{lemma}

The proof is straight forward with the fact that if $m\notin Crit$ then $|\cos(x)| \geq N^{-\tilde{\delta}}$.

\begin{definition}
Consider the cone $\Delta_{\tilde{\delta}} = \{(u,v) \in \R^2: N^{\tilde{\delta}}|u| \geq |v|\}$. If an adapted vector field $(\gamma, X)$ is tangent to this cone we say that it is a {\bf$\tilde{\delta}$-good} adapted vector field. Otherwise we say that it is {\bf $\tilde{\delta}$-bad}.
\end{definition}

\begin{lemma}
\label{esselemmaaqui}
For $N$ sufficiently large and for every $\tilde{\delta}$-good adapted vector field $(\gamma,X)$
\begin{eqnarray*}
|\sin(\theta_X(m))|> N^{-4\tilde{\delta}}& \forall m\notin Crit. 
\end{eqnarray*}
Furthermore, for a $\tilde{\delta}$-good adapted field $(\gamma,X)$, if $m\notin Crit$ then $\|Df(m).X_m\| \geq N^{1-6\tilde{\delta}}$ 
\end{lemma}
\begin{proof}
Recall that $v_m = (1, \Omega(m))$ and suppose that $\Omega(m) >0$. Let $b_m = (1, N^{\tilde{\delta}})$ and consider the triangle formed by the points $0$, $b_m$ and $v_m$, see figure \ref{image4}. Denote by $\measuredangle(u,v)$ the angle between two vectors $u,v\in \R^2$. By the  law of sines
\begin{equation}
\label{prova1}
\displaystyle \frac{\sin ( \measuredangle (v_m, b_m))}{\|v_m - b_m\|} = \frac{ \sin( \measuredangle(v_m - b_m, b_m))}{\|v_m\|}.
\end{equation}

 For a good adapted field $(\gamma,X)$, it holds $| \sin(\theta_X(m))| \geq |\sin ( \measuredangle(v_m,b_m))|$. Recall that $m\notin Crit$, by lemma \ref{lemmaanterior} we have $N^{1-2\tilde{\delta}} \leq |\Omega(m)| \leq N$. Observe that 
\begin{equation}
\label{prova2}
\sin( \measuredangle(v_m - b_m, b_m)) = \frac{1}{\|b_m\|}.
\end{equation}
\begin{figure}
\centering
\includegraphics[scale=0.4]{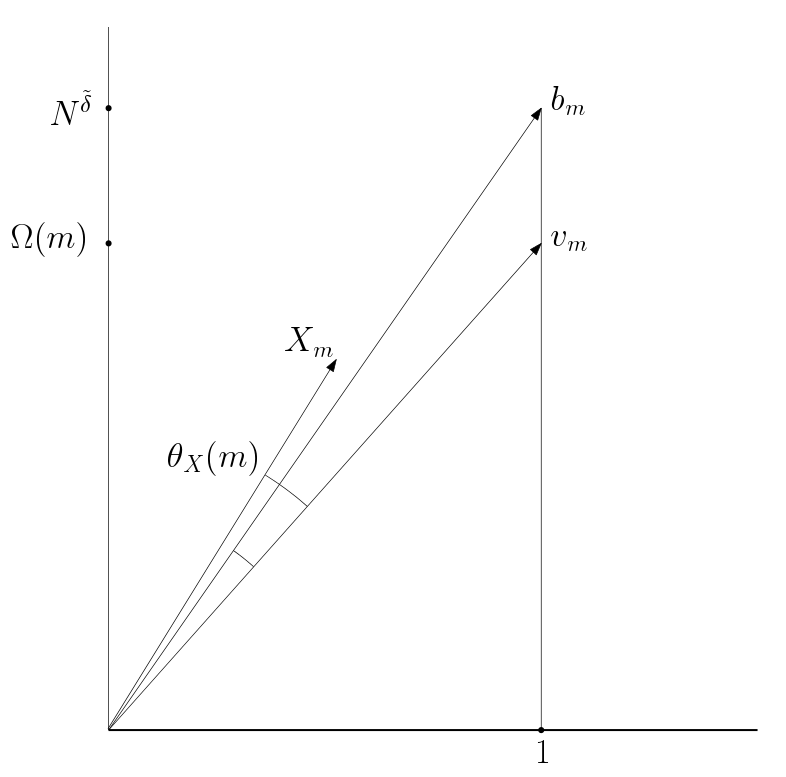}
\caption{The triangle formed by $0$, $b_m$ and $v_m$}
\label{image4}
\end{figure}
By (\ref{prova1}) and (\ref{prova2}), for $N$ large enough we obtain
\[
|\sin(\theta_X(m))| \geq \frac{|\Omega(m)|-N^{\tilde{\delta}}}{\sqrt{1+N^{2\tilde{\delta}}}.\sqrt{1+\Omega(m)^2}} \geq N^{-4\tilde{\delta}}.
\]
It follows from this inequality and lemma \ref{lemmaanterior}, that for a $\tilde{\delta}$-good adapted field $(\gamma,X)$, if $m\notin Crit$ then $\|Df(m).X_m\| \geq N^{1-6\tilde{\delta}}$. If $\Omega (m)<0$ we can obtain the same estimate taking $b_m = (1, -N^{\tilde{\delta}})$.
\end{proof}

\begin{proposition}
\label{sogood}
For $N$ sufficiently large if $(\gamma, X )$ is a $\tilde{\delta}$-good adapted vector field then $E(\gamma,X) \geq (1-7\tilde{\delta}) \log N$.

\end{proposition}
\begin{proof}
Recall that for a $u$-curve $\frac{d\gamma}{dt}(t) = (\alpha(\gamma(t)), e^u)$ and $|P_x(\alpha(\gamma(t))-(1,0)|\leq \lambda^{2N}$. In particular, using that $l(Crit) \leq 8N^{-\tilde{\delta}}$, for $N$ large enough the measure of $\gamma \cap Crit $ is smaller than $10N^{-\tilde{\delta}}|\gamma|$.

The previous lemma give us an estimate for points outside the critical strip. For points inside the critical strips we use that $\|Df|_{E^c}\| \geq (2N)^{-1}$. Thus for $N$ large enough we get

$$
\begin{array}{rcl}
|\gamma|E(\gamma,X) & = & \int_{\gamma \cap Crit} \log \|Df(m).X_m\|d\gamma + \int_{\gamma \cap Crit^c} \log\|Df(m).X_m\|d\gamma \\
&\geq & \left( 1-\frac{10}{N^{\tilde{\delta}}}\right).(1-6\tilde{\delta})\log N|\gamma| - \left(\frac{10}{N^{\tilde{\delta}}}\right).\log  2N|\gamma| \geq (1-7\tilde{\delta})\log N|\gamma|.
\end{array}
$$
\end{proof}

Recall that $f^k\circ \gamma = \gamma_1^k \ast\cdots \ast \gamma_{[\mu^{2Nk}]}^k \ast \gamma_{[\mu^{2Nk}]+1}^k$ and define

\begin{eqnarray}
\label{goodbad}
G_k=G_k(\gamma,X)&=&\left\{ 1\leq j \leq [\mu^{2Nk}]: \left(\gamma_j^k,\frac{f_*^kX}{\|f_*^kX\|}\right)\textrm{ is $\tilde{\delta}$-good.}\right\},\\
B_k=B_k(\gamma,X) &=& \left\{ 1\leq j \leq [\mu^{2Nk}]: \left(\gamma_j^k,\frac{f_*^kX}{\|f_*^kX\|}\right)\textrm{ is $\tilde{\delta}$-bad.}\right\}.
\end{eqnarray}

\begin{lemma}
\label{lemmabom}
For $N$ sufficiently large, if $(\gamma,X)$ is a $\tilde{\delta}$-good adapted field and $f^{-1}(\gamma^1_j) \cap Crit =\emptyset$, then $(\gamma^1_j, \frac{f_*X}{\|f_*X\|})$ is also $\tilde{\delta}$-good. 
\end{lemma}

\begin{proof}
Let $m\notin Crit$ and $v\in (-N^{\tilde{\delta}}, N^{\tilde{\delta}})$. It is verified
$$
Df(m).(1,v) = (\Omega(m) - v, 1).
$$
By lemma \ref{esselemmaaqui} 
$$
|\Omega(m)-v| \geq |\Omega(m)|-|v| \geq N^{1-2\tilde{\delta}} - N^{\tilde{\delta}},
$$
which is arbitrarily large as $N$ grows. This implies that the vector $(\Omega(m) - v, 1)$ is inside the cone $\Delta_{\tilde{\delta}}$, because it will be very close to the $x$ axis. 
\end{proof}

The next lemma is the same as lemma $6$ in \cite{bergercarrasco2014}. 

\begin{lemma}
\label{lemmaruim}
For $N$ sufficiently large, for every $\tilde{\delta}$-bad adapted vector field, there is a strip $S_X$ of length $\pi$ such that if $f^{-1}(\gamma^1_j)\subset S_X $ then $(\gamma^1_j, \frac{f_*X}{\|f_*X\|})$ is $\tilde{\delta}$-good.
\end{lemma}

Let $\eta_N = \frac{5}{\pi N^{\tilde{\delta}}}$. The following proposition is analoguous to proposition 4 in \cite{bergercarrasco2014}.

\begin{proposition}
For $N$ large enough, for every $\tilde{\delta}$-bad adapted field
$$
\begin{array}{rcl}
\# G_1 \geq \frac{1}{3} \mu^{2N} & \textrm{ and }& \#B_1 \leq \frac{2}{3}\mu^{2N}.
\end{array}
$$
For every $\tilde{\delta}$-good adapted field
$$
\begin{array}{rcl}
\# G_1 \geq \left(1-\eta_N \right) \mu^{2N} & \textrm{ and }& \#B_1 \leq \left(\eta_N\right)\mu^{2N}.
\end{array}
$$

\end{proposition}

\begin{proof}
Using lemma \ref{lemmaruim} there is a strip $S_X$ of length $\pi$ such that if $f^{-1}(\gamma^1_j) \subset S_X$, this represents almost half of the pieces $\gamma^1_j$, for $N$ large enough we conclude the first part of the proposition. For the second part we use lemma \ref{lemmabom} and the fact that $l(Crit) \leq 8N^{-\tilde{\delta}}$ and by a similar argument, for $N$ large, it holds the second part of the proposition.
\end{proof}

Now in general for any $k\in \N$,

$$
\begin{array}{rcl}
\#G_{k+1} &\geq& (1-\eta_N)\mu^{2N} \#G_k + \frac{1}{3}\mu^{2N}\#B_k\\
\#B_{k+1}& \leq & \eta_N \mu^{2N} \#G_k + \frac{2}{3}\mu^{2N}\#B_k
\end{array}
$$

\begin{lemma}
\label{goodandbad}
For any $K\geq 1$, if $N$ is large enough then for any $k\geq 0$ and any $\tilde{\delta}$-good adapted vector field $(\gamma, X)$, it is verified $\#G_k \geq K\#B_k$.
\end{lemma}
\begin{proof}
Since $(\gamma,X)$ is $\tilde{\delta}$-good then $B_0=0$ and $\#G_0=1 > K.\#B_0$. By our previous remark if $N$ is large enough then it is also valid for $k=1$, let us suppose that it is valid for $k$ and prove it for $k+1$.
\[\arraycolsep=1.2pt\def\arraystretch{2.5}
\begin{array}{rclcl}
\displaystyle\frac{\#B_{k+1}}{\#G_{k+1}}&\leq&\displaystyle \frac{\eta_N \mu^{2N} \#G_k + \frac{2}{3}\mu^{2N}\#B_k }{(1-\eta_N)\mu^{2N} \#G_k + \frac{1}{3}\mu^{2N}\#B_k}& \leq&\displaystyle \frac{\eta_N \mu^{2N} \#G_k + \frac{2}{3}\mu^{2N}K^{-1}\#G_k}{(1-\eta_N)\mu^{2N} \#G_k}\\
& = &\displaystyle \frac{\eta_N + \frac{2}{3}K^{-1}}{1-\eta_N} \leq  \displaystyle\frac{\eta_N + \frac{2}{3}K^{-1}}{1-\eta_N}&<& \frac{3}{4K}.
\end{array}
\]

Where the last inequality holds for $N$ large. Thus $\#G_{k+1} > \frac{4K}{3}\#B_{k+1}>K\#B_{k+1}.$ 
\end{proof}

Now we can get the estimate on the Lyapunov exponent that we wanted. 

\begin{lemma}
\label{mencionado}
For $N$ large enough and for every $\tilde{\delta}$-good adapted vector field $(\gamma,X)$ and every $n\geq 1$ we have
$$
\frac{I^{\gamma,X}_n}{n} \geq (1-10\tilde{\delta})\log N.
$$
\end{lemma}

\begin{proof}
 Fix $K>0$ large enough such that $K^{-1} < \tilde{\delta}$. Let $(\gamma, X)$ be a $\tilde{\delta}$-good adapted vector field, by the previous lemma $\#G_k>\frac{1}{1+K^{-1}} \mu^{2Nk}$. Using the formula given by lemma \ref{formula1}, the estimate obtained for $\tilde{\delta}$-good adapted vector field in proposition \ref{sogood} and for every $\tilde{\delta}$-bad adapted vector field using that $\|Df|_{E^c}\| \geq (2N)^{-1}$, we conclude

\[\arraycolsep=1.2pt\def\arraystretch{2.3}
\begin{array}{rcl}
\frac{I_n^{\gamma,X}}{n} &\geq&\displaystyle \frac{1}{n} \sum_{k=0}^{n-1} \frac{(1-2\lambda^{2N})}{\mu^{2Nk}}\left(\#G_k.(1-7\tilde{\delta})\log N- (\#B_k + 1) \log 2N\right)\\
&\geq&\displaystyle  \frac{1}{n} \sum_{k=0}^{n-1}(1-2\lambda^{2N})\left(\frac{1}{1+K^{-1}}.(1-7\tilde{\delta})\log N- K^{-1}\log2- K^{-1}\log N + \frac{\log 2N}{\mu^{2Nk}} \right)\\
&\geq& (1-10\tilde{\delta})\log N
\end{array}
\]

For $N$ large enough.
\end{proof}

With this lemma we can prove the estimate of proposition \ref{estimateprop} for $f_N$.

\begin{corollary}
\label{estimatefibered}
For $\delta>0$, if $N$ is large enough then almost every point has a Lyapunov exponent on the center direction greater than $(1-\delta) \log N$ for $f_N$.
\end{corollary}
\begin{proof}
Take $\tilde{\delta} = \frac{\delta}{30}$ and let $N$ be large enough such that the previous lemma holds. Thus we can take $C = (1- 10 \tilde{\delta}) \log N = (1-\frac{\delta}{3}) \log N$, where $C$ is the constant from proposition \ref{exp1}. Assume that $N$ is large enough such that $(1-2\lambda^{2N})(1-\frac{\delta}{3}) > ( 1-\delta)$. The result follows from proposition \ref{exp1}.
\end{proof}

\subsection{Robustness of the estimate}
In this section we prove proposition \ref{estimateprop}. For a $C^1$-curve $\gamma$ we will denote by $Leb_{\gamma}$ the Lebesgue measure induced by the Riemaniann metric in the curve. Recall that for each $N\in \N$ we denote by $\mathcal{U}_N \subset \mathrm{Diff}^2_{Leb}(\T^4)$ a $C^2$-neighborhood of $f_N$.

\begin{lemma}
\label{manyconsiderationsnotfibered}
For $\varepsilon_1>0$ small, if $N$ is large and $\mathcal{U}_N$ is small enough then for every $g\in U$ and for all unit vectors $v^s\in E^s_g$, $v^c\in E^c_g$ and $v^u\in E^u_g$, the following holds:

\begin{enumerate}
\item $e^{-\varepsilon_1} \lambda^{2N}  \leq  \|Dg(v^s)\|  \leq  e^{\varepsilon_1} \lambda^{2N}$;
\item $e^{-\varepsilon_1} \mu^{2N}  \leq  \|Dg(v^u)\| \leq  e^{\varepsilon_1} \mu^{2N}$;
\item $\frac{1}{2N}  \leq  \|Dg(v^c)\|  \leq  2N;$
\item $\|D^2g^{-1}\|\leq 2N$ and $\|D^2g\|\leq 2N$;
\item $E^c_g$ is $\frac{1}{2}$-H\"older.
\end{enumerate}

\end{lemma}

\begin{proof}
The only statement that does not follow directly from $C^2$-continuity for $N$ large enough is $(5)$. Observe that 
\[
e^{\varepsilon_1} \lambda^{2N} < (2N)^{-1} e^{-\frac{\varepsilon_1}{2}} \lambda^N.
\] 
Hence, by theorem \ref{regularitycenter} it follows that $E^c_g$ is $\frac{1}{2}$-H\"older.
\end{proof}

\begin{definition}
A $u$-curve for $g$ is a $C^1$-curve $\gamma=(\gamma_x,\gamma_y,\gamma_z,\gamma_w):[0,2\pi] \to M$ tangent to $E^u_g$ and such that $\left|\frac{d \gamma_x}{dt}(t)\right| = 1$, $\forall t\in [0,2\pi]$. For every $k\geq 0$ there exists an integer $N_k = N_k(\gamma)$ such that the curve $g^k\circ \gamma$ can be writen as 
$$
g^k \circ \gamma = \gamma_1^k \ast\cdots \ast \gamma_{[\mu^{2Nk}]}^k \ast \gamma_{N_k+1}^k
$$
where $\gamma_j^k$ for $j=1,\cdots, N_k$, are $u$-curves and $\gamma_{N_k+1}^k$ is a segment of $u$-curve. 

\end{definition}

Observe that this definition of an $u$-curve is different from the one given in definition \ref{ucurve}. The advantage of definition \ref{ucurve} is that during the calculations we do not have to deal with bounded distortion estimates. Since for the general case it is natural to appear bounded distortion estimates, see lemma \ref{boundeddist}, we just normalize the curve on the $x$-direction in the previous definition. 

\begin{lemma}[\cite{bergercarrasco2014}, Corollary $5$]
\label{length}
For $\varepsilon_2>0$ small, if $N$ is large and $\mathcal{U}_N$ is small enough then for every $g\in \mathcal{U}_N$ and unit vector $v^u\in E^u_{g,m}$, it holds that
\[
|P_x(D\pi_1.v^u)| \in [(\lambda^N(\|P_x(e^u) - 3\lambda^N\|), (\lambda^N(\|P_x(e^u) + 3\lambda^N\|)].
\]
In particular, for any two $u$-curves $(\gamma, \gamma^{\prime})$ satisfy:
\[
e^{-\varepsilon_2} l(\gamma) \leq l(\gamma^{\prime} ) \leq e^{\varepsilon_2}l(\gamma).
\]

\end{lemma}

Define similarly as in definition \ref{adaptedfield} an adapted field $(\gamma, X)$. Also define the unstable jacobian of $g^k$ as
\[
J^{uu}_{g^k}(m) = | \det Dg^k(m)|_{E^{uu}_g}|, \textrm{ $\forall m \in \T^4$.}
\]
By item $2$ of lemma \ref{manyconsiderationsnotfibered}, for $g\in \mathcal{U}_N$ and for every $m\in \T^4$
\[
e^{-\varepsilon_1}\lambda^{2N} \leq J^{uu}_{g^{-1}}(m) \leq e^{\varepsilon_1}\lambda^{2N}.
\]

The proof of the next lemma is classical and can be found in \cite{bergercarrasco2014}, lemma $8$.
\begin{lemma}[Bounded distortion]
\label{boundeddist}
For $\varepsilon_3>0$ small, if $N$ is large and $\mathcal{U}_N$ is small enough, for every $g\in \mathcal{U}_N$ and any $u$-curve $\gamma$ for $g$, for every $k\geq 0$, it holds
\[
\forall m, m^{\prime} \in \gamma, \textrm{  } e^{-\varepsilon_3} \leq \frac{J^{uu}_{g^{-k}}(m)}{J^{uu}_{g^{-k}}(m^{\prime})}\leq e^{\varepsilon_3}.
\] 
\end{lemma}

This lemma implies that for $g\in \mathcal{U}_N$ and for any $u$-curve $\gamma$ for $g$, if $A\subset \gamma$ is any measurable set, for every $k\geq 0$, it holds
\[
e^{-\varepsilon_3} \frac{Leb(A)}{Leb(\gamma)} \geq \frac{Leb(g^{-k}(A))}{Leb(g^{-k}(\gamma))} \leq e^{\varepsilon_3} \frac{Leb(A)}{Leb(\gamma)}.
\]

Let $(\gamma,X)$ be an adapted field, define
\[
I_n^{\gamma,X} = \frac{1}{|\gamma|} \int_{\gamma} \log \|Dg^n.X\| d\gamma.
\]
For the fibered case, proposition \ref{exp1} gives us precise estimates for the Lyapunov exponent along the center direction. In the general case we have the following proposition. 
\begin{proposition}
\label{estimaterob}
Suppose that there exists $C>0$ with the following property: for every $u$-curve $\gamma$ there exists an adapted vector field $(\gamma,X)$ for $g$ and for all $n> 0$ large enough
\[
\frac{I_n^{\gamma,X}}{n} >C.
\]
Then the map $g$ has a positive exponent in the center direction greater than $e^{-2\varepsilon_3}C$ for Leb-almost every point.
\end{proposition}
 
\begin{proof}
The new ingredient in the proof is the bounded distortion estimates. Suppose not, then there exists a measurable set $B$ with positive measure such that every point in $B$ has exponents in the center direction strictly smaller than $e^{-2\varepsilon_3}C$. By the absolute continuity of the unstable foliation, there is an unstable manifold $\gamma$ that intersects $B$ on a set of positive measure, for the Lebesgue measure of $\gamma$. Let $b\in \gamma\cap B$ be a density point and take $\gamma_k = g^{-k} \circ \beta_k$, where $\beta_k$ is a $u$-curve with $\beta_k(0) = g^k(b)$. We have that $l(\gamma_k)\to 0$ and by bounded distortion, lemma \ref{boundeddist} 
\[
\frac{Leb(\gamma_k \cap B )}{Leb(\gamma_k)} \longrightarrow 1.
\]

Take $k$ large enough such that
\[
\frac{Leb(\gamma_k \cap B^c )}{Leb(\gamma_k)}< \frac{e^{-2\varepsilon_3} (e^{\varepsilon_3} - 1) C}{2\log 2N}.
\]
Using bounded distortion again, for any $m^k \in g^k(\gamma_k)$
\[
J^{uu}_{g^{-k}}(m^k) \geq \frac{Leb(\gamma_k)}{Leb(g^k(\gamma_k))} e^{-\varepsilon_3}.
\]

Define $\chi_k(m) = \displaystyle \limsup_{n\to +\infty} \frac{1}{n} \log \|Dg^n(g^k(m)) . X_{g^k(m)}\|$, where $X$ is the vector field such that $(\beta_k,X)$ verifies the hypothesis of the lemma.

\[\arraycolsep=1.2pt\def\arraystretch{2}
\begin{array}{rcl}
\displaystyle \int_{\gamma_k} \chi_k d\gamma_k & = & \displaystyle \int_{g^k(\gamma_k)} \chi_k \circ g^{-k} J^{uu}_{g^{-k}} d(g^k(\gamma_k))\\
& \geq & \displaystyle e^{-\varepsilon_3} \frac{Leb(\gamma_k)}{Leb(g^k(\gamma_k))} \int_{g^k(\gamma_k)} \chi_k \circ g^{-k} d(g^k(\gamma_k)) \geq e^{-\varepsilon_3} C Leb(\gamma_k).
\end{array}
\]

On the other hand,

\[\arraycolsep=1.2pt\def\arraystretch{2}
\begin{array}{rcl}
\displaystyle \int_{\gamma_k} \chi_k d\gamma_k & = & \displaystyle \int_{\gamma_k \cap B} \chi_k d\gamma_k + \int_{\gamma_k \cap B^c} \chi_k d\gamma_k\\
 & \leq & \displaystyle e^{-2\varepsilon_3} C Leb(\gamma_k) + \frac{\log 2N e^{-2\varepsilon_3}(e^{\varepsilon_3}-1)C Leb(\gamma_k)}{2 \log 2N}\\
 &< & e^{-\varepsilon_3} C Leb(\gamma_k).
\end{array}
\]
Which is a contradiction.
\end{proof}

Denote by 
\[
E(\gamma,X) = \frac{1}{|\gamma|} \displaystyle \int_{\gamma} \log \|Dg(m).X_m\| d\gamma(m),
\]
where $(\gamma,X)$ is an adapted field. For $X$ a vector field on $\gamma$ define 
\[
\widetilde{X}(m) = \frac{\pi_1(X(m))}{\|\pi_1(X(m))\|}.
\]
\begin{definition}
An adapted field $(\gamma,X)$ is $\tilde{\delta}$-good if for every $m\in \gamma $, $\widetilde{X}(m) \in \Delta_{\tilde{\delta}}$.
\end{definition}

If $\mathcal{U}_N$ is small enough then the center leaves are very close to the horizontal tori, very similar to the proof of proposition \ref{sogood} we obtain:

\begin{proposition}
\label{sogood2}
For $N$ large and $\mathcal{U}_N$ small enough, for all $g\in \mathcal{U}_N$ and $(\gamma, X)$ an $\tilde{\delta}$-good adapted field for $g$, it is verified that $E(\gamma, X) \geq (1-8\tilde{\delta}) \log N.$
\end{proposition}

Recall that for $k\geq 0$ and a $u$-curve $\gamma$ the number $N_k= N_k(\gamma)$ was the maximum number of $u$ curves that subdivide $g^k \circ \gamma$. For an adapted field $(\gamma,X)$ define $Y^k = \frac{g^k_*X}{\|g^k_*X\|}$. The following lemma is the analogous to lemma \ref{bclemma2}.

\begin{lemma}[\cite{bergercarrasco2014}, Lemma $9$]
For $N$ large and $\mathcal{U}_N$ small enough, let $g\in \mathcal{U}_N$ and $(\gamma,X)$ be an adapted field for $g$. For $k\geq 0$, every possible pair $(\gamma_j^k,Y^k|_{\gamma_j^k})$, with $1\leq j \leq N_k(\gamma)$ is an adapted field.
\end{lemma}

Similar to lemma \ref{formula1}, Berger and Carrasco proved the following formula, see section $6$ of \cite{bergercarrasco2014}.
\begin{lemma}
\label{formula2}
For every adapted field $(\gamma,X)$ and any $n\in \N$
\[
I_n^{\gamma,X} = \displaystyle \sum_{k=0}^{n-1} \left( R_k+ \sum_{j=0}^{N_k} \frac{1}{|\gamma|} \int_{\gamma_j^k} \log \| Dg (m).Y^k_m\| J^{uu}_{g^{-k}} d \gamma_j^k \right),
\]

where $R_k=\frac{1}{|\gamma|} \int_{\gamma_{N_k+1}^k} \log \| Dg (m).Y^k_m\| J^{uu}_{g^{-k}} d \gamma_{N_k+1}^k$. 
\end{lemma}

We remark that this formula and the formula obtained in lemma \ref{formula1} are obtained in the same way, just by using the change of variables formula multiple times. The difference in this one is that we keep the unstable jacobian in the formula. 
As a consequence of this formula we obtain

\begin{eqnarray}
\label{aiaiai}
I_n^{\gamma,X} \geq \displaystyle \sum_{k=0}^{n-1} \left( R_k+ \sum_{j=0}^{N_k} \min_{\gamma_j^k} ( J^{uu}_{g^{-k}}. E( \gamma_j^k, Y^k)) \right).
\end{eqnarray}

Observe that
\[
\displaystyle |R_k| \leq \frac{(e^{-\varepsilon_1} \mu)^{2Nk} \log 2N }{\lambda^N(1-2\lambda^N) \|P_x(e^u)\|} \xrightarrow{k\rightarrow + \infty} 0. 
\]
Hence
\[
\displaystyle \frac{1}{n} \sum_{k=0}^{n-1} |R_k| \longrightarrow 0.
\]

For $(\gamma,X)$ an adapted field we define similarly as in the previous section the sets $G_k = G_k(\gamma, X)$ and $B_k= B_k(\gamma,X)$. The key lemma is the next one which is analogous to lemma \ref{goodandbad}.
\begin{lemma}
\label{lemmabolado}
For $K\geq 1$, for $N$ large and $\mathcal{U}_N$ small enough, for every $g\in \mathcal{U}_N$ and every $(\gamma,X)$ a $\tilde{\delta}$-good adapted field it holds that
\[
\displaystyle \sum_{j\in G_k} \min_{\gamma_j^k} J^{uu}_{g^{-k}} \geq K \sum_{j\in B_k} \max_{\gamma_j^k} J^{uu}_{g^{-k}}.
\]

\end{lemma}

The proof uses the next lemma, which is analogous to lemmas \ref{lemmabom} and \ref{lemmaruim}.
\begin{lemma}
\label{lemma12}
For $N$ large and $U$ small enough, for every $g\in \mathcal{U}_N$, every adapted field $(\gamma, X)$
\begin{enumerate}
\item If $(\gamma,X)$ is a $\tilde{\delta}$-good adapted field and if $j$ is so that $g^{-1} \gamma_j^1$ does not intersect the strip $Crit$, then the field $(\gamma_j^1, \frac{g_*X}{\|g_*X\|})$ is $\tilde{\delta}$-good.
\item If $(\gamma, X)$ is $\tilde{\delta}$-bad, there exists a strip $S$ of length $\pi$ such that for every $j$ satisfying $g^{-1} \gamma_j^1\subset S$, the field $(\gamma_1^j, \frac{g_*X}{\|g_*X\|})$ is $\tilde{\delta}$-good.
\end{enumerate}
\end{lemma}
The proof of this lemma is similar to the proof of  lemma $12$ in \cite{bergercarrasco2014} and uses the estimate obtained in lemma \ref{esselemmaaqui}. 

\begin{proof}[ Proof of lemma \ref{lemmabolado}.]

We follow exactly Berger-Carrasco's proof with the constants we chose and taking $\eta_N = \frac{5}{\pi N^{\delta}}$. The proof goes by induction, it is valid for $k=0$ and suppose it is true for $k$. Using lemmas \ref{length} and \ref{lemma12}, following exactly the same proof of Berger and Carrasco, we obtain
\[
\displaystyle \sum_{j\in G_{k+1}} \min_{\gamma_j^k} J^{uu}_{g^{-k-1}} \geq e^{-(\varepsilon_2 + \varepsilon_3)}(1-\eta_N) \sum_{j\in G_k} \min_{\gamma_j^k} J^{uu}_{g^{-k}}.
\]

It is also obtained
\[
\displaystyle \sum_{j\in B_{k+1}} \max_{\gamma_j^k} J^{uu}_{g^{-k-1}} \leq \left(e^{\varepsilon_2 +2 \varepsilon_3} \eta_N + \frac{2.2}{3.K} e^{\varepsilon_3}\right).\left(\sum_{j\in G_k} \min_{\gamma_j^k} J^{uu}_{g^{-k}} \right) + \lambda^{\frac{N}{2}} e^{\varepsilon_3}.
\]

Thus
\[
\displaystyle \frac{\sum_{j\in B_{k+1}} \max_{\gamma_j^k} J^{uu}_{g^{-k-1}}}{\sum_{j\in G_{k+1}} \min_{\gamma_j^k} J^{uu}_{g^{-k-1}}} \leq \frac{e^{\varepsilon_2 +2 \varepsilon_3} \eta_N + \frac{2.2}{3.K} e^{\varepsilon_3}}{e^{-(\varepsilon_2 + \varepsilon_3)}(1-\eta_N)} + \frac{\lambda^{\frac{N}{2}}}{e^{-(\varepsilon_2 + \varepsilon_3)}(1-\eta_N)}< \frac{1}{K},
\]
since we fixed $\varepsilon_2$ and $\varepsilon_3$ very small, for $N$ large enough we obtain the last inequality.
\end{proof}

From now on we fix $K> (\tilde{\delta})^{-1}$ and assume that $N$ is large and $\mathcal{U}_N$ is small enough such that lemma \ref{lemmabolado} holds.

\begin{lemma}
\label{someestimate}
For $N$ large and $\mathcal{U}_N$ small enough, for every $g\in \mathcal{U}_N$, every adapted field $(\gamma, X)$ and $k\geq 0$, it holds
\[
e^{-(\varepsilon_2 + \varepsilon_3)} \leq \displaystyle \sum_{j\in G_k} \min_{\gamma_j^k} J^{uu}_{g^{-k}} + \sum_{j\in B_k} \max_{\gamma_j^k} J^{uu}_{g^{-k}} \leq e^{2(\varepsilon_2 + \varepsilon_3)}.
\]
\end{lemma}

\begin{proof}
Of course the lemma is true for $k=0$. Following the same steps as the proof of lemma $11$ in \cite{bergercarrasco2014}, one obtains

\[\arraycolsep=1.2pt\def\arraystretch{2}
\begin{array}{lcl}
1 & = & \displaystyle \frac{1}{|\gamma|} \int_{\gamma} d\gamma = \frac{1}{|\gamma|}\sum_{j=1}^{N_k+1} \int_{\gamma^k_j}J^{uu}_{g^{-k}}d\gamma_j^k\\
& \geq & \displaystyle \sum_{j\in G_k} \frac{|\gamma_j^k|}{|\gamma|} \min_{\gamma_j^k} J^{uu}_{g^{-k}} + e^{-\varepsilon_3} \sum_{j\in B_k}\frac{|\gamma_j^k|}{|\gamma|} \max_{\gamma_j^k} J^{uu}_{g^{-k}}- \int_{\gamma_{N_k+1}^k} \max_{\gamma_{N_k+1}^k} J^{uu}_{g^{-k}} d\gamma_{N_k+1}^k \\
& \Rightarrow & \displaystyle 1 \geq e^{-(\varepsilon_2 + \varepsilon_3)} \left(\left(\sum_{j\in G_k}  \min_{\gamma_j^k} J^{uu}_{g^{-k}}\right) +  \left(\sum_{j\in B_k} \max_{\gamma_j^k} J^{uu}_{g^{-k}}\right)- \frac{(e^{-\varepsilon_1}. \mu)^{-2Nk}}{\lambda^N(1-2\lambda^N) \|P_x(e^u)\|}\right).
\end{array}
\]

For $N$ large enough 
\[
1+\frac{e^{-(\varepsilon_2 + \varepsilon_3)}(e^{-\varepsilon_1}. \mu)^{-2Nk}}{\lambda^N(1-2\lambda^N) \|P_x(e^u)\|}< e^{\varepsilon_2 + \varepsilon_3}.
\]
Hence
\[
\left(\sum_{j\in G_k}  \min_{\gamma_j^k} J^{uu}_{g^{-k}}\right) +  \left(\sum_{j\in B_k} \max_{\gamma_j^k} J^{uu}_{g^{-k}}\right)\leq e^{2(\varepsilon_2 + \varepsilon_3)}.
\]
Similarly one obtains the other inequality.
\end{proof}

We remark that this lemma for the fibered case is immediate, since in this case $\#G_k + \#B_k =[\mu^{2Nk}]$ and by the way we parametrize $u$-curves for the fibered case, $J^{uu}_{f^{-k}} = \mu^{-2Nk}$. Since the calculations for the fibered case are more direct, the application of this lemma is hidden inside the proof of lemma \ref{mencionado}. For the general case we use this lemma to obtain inequality (\ref{finalequality}) below. This is done in the following way. By lemmas \ref{lemmabolado} and \ref{someestimate},

\[
e^{-2(\varepsilon_2 + \varepsilon_3)} \leq ( 1 + K^{-1} )\displaystyle \sum_{j\in G_k} \min_{\gamma_j^k} J^{uu}_{g^{-k}},
\]
which implies that 
\begin{equation}
\label{finalequality}
\frac{e^{-2(\varepsilon_2 + \varepsilon_3)}}{1+K^{-1}}  \leq \sum_{j\in G_k} \min_{\gamma_j^k} J^{uu}_{g^{-k}}.
\end{equation}

\begin{proposition}
\label{agoravai}
For $N$ large and $\mathcal{U}_N$ small enough, for every $g\in \mathcal{U}_N$, any $\tilde{\delta}$-good adapted field $(\gamma,X)$ and every $k\geq 0$, it holds
\[
\displaystyle \sum_{j=0}^{N_k} \min_{\gamma_j^k} (J^{uu}_{g^{-k}}.E(\gamma_j^k,Y^k)) \geq (1-12\tilde{\delta}) \log N.
\]
\end{proposition}

\begin{proof}
We have
\[
\displaystyle \sum_{j=0}^{N_k} \min_{\gamma_j^k} (J^{uu}_{g^{-k}}.E(\gamma_j^k,Y^k)) = \sum_{j\in G_k} \min_{\gamma_j^k} (J^{uu}_{g^{-k}}.E(\gamma_j^k,Y^k))+ \sum_{j\in B_k} \min_{\gamma_j^k} (J^{uu}_{g^{-k}}.E(\gamma_j^k,Y^k)) 
\]

By lemmas \ref{lemmabolado} and \ref{someestimate} and proposition \ref{sogood2} we obtain
\[\arraycolsep=1.2pt\def\arraystretch{2.2}
\begin{array}{rcl}
\displaystyle \sum_{j=0}^{N_k} \min_{\gamma_j^k} (J^{uu}_{g^{-k}}.E(\gamma_j^k,Y^k)) & \geq & (1-8\tilde{\delta})\log N\displaystyle \sum_{j\in G_k} \min_{\gamma_j^k} J^{uu}_{g^{-k}} - \log 2N \sum_{j\in B_k} \min_{\gamma_j^k} J^{uu}_{g^{-k}}\\
& \geq & \displaystyle \left((1-8\tilde{\delta}) - \frac{\log 2N}{K} \right) \sum_{j\in G_k} \min_{\gamma_j^k} J^{uu}_{g^{-k}}\\
& \geq & \displaystyle \frac{e^{-2(\varepsilon_2 + \varepsilon_3)}(1-10\tilde{\delta}) \log N}{1+K^{-1}}> (1-12\tilde{\delta})\log N.
\end{array}
\]
\end{proof}

\begin{proof}[Proof of Proposition \ref{estimateprop}.]

Take $\tilde{\delta} = \frac{\delta}{15}$. By proposition \ref{agoravai}, for $N$ large and $\mathcal{U}_N$ small enough, for $g\in \mathcal{U}_N$ and any $\tilde{\delta}$-good adapted field $(\gamma, X)$, for $g$, it holds that
 
\[
\displaystyle \sum_{j=0}^{N_k} \min_{\gamma_j^k}( J^{uu}_{g^{-k}}.E(\gamma_j^k, Y^k)) \geq (1-12\tilde{\delta}) \log N.
\]

Using inequality (\ref{aiaiai}), for $n$ large enough  
\[
\frac{I^{\gamma,X}_n}{n} \geq (1-14\tilde{\delta})\log N.
\]
Since we could have chosen $\varepsilon_3>0$ small enough such that $e^{-\varepsilon_3} (1-14\tilde{\delta}) \geq (1-15\tilde{\delta})$ by proposition \ref{estimaterob}, almost every point has a Lyapunov exponent for $g$ in the center direction larger than 
\[
(1-15\tilde{\delta}) \log N = (1-\delta) \log N.
\]
All we did is also valid for $g^{-1}$, if $\mathcal{U}_N$ is small enough, thus almost every point has a negative Lyapunov exponent in the center direction smaller than $-(1-\delta) \log N$.
\end{proof}

\address

\end{document}